\documentclass[12pt]{amsart}

\usepackage{amsmath,amsfonts,amssymb,amsthm}
\usepackage{graphicx}
\usepackage{color,colordvi}
\usepackage{hyperref}

\def\softd{{\leavevmode\setbox1=\hbox{d}%
\hbox to 1.05\wd1{d\kern-0.4ex{\char039}\hss}}}
\def\softt{{\leavevmode\setbox1=\hbox{t}%
\hbox to \wd1{t\kern-0.6ex{\char039}\hss}}}

\newcommand{\C}{\mathbb{C}}
\newcommand{\N}{\mathbb{N}}
\newcommand{\R}{\mathbb{R}}
\newcommand{\T}{\mathbb{T}}
\newcommand{\Z}{\mathbb{Z}}

\newcommand{\GG}{\mathcal{G}}
\newcommand{\HH}{\mathcal{H}}
\newcommand{\LL}{\mathcal{L}}
\newcommand{\NN}{\mathcal{N}}
\newcommand{\OO}{\mathcal{O}}
\newcommand{\RR}{\mathcal{R}}

\newcommand{\cY}{\mathcal{Y}}
\newcommand{\cZ}{\mathcal{Z}}
\newcommand{\ee}{\mathrm{e}}
\newcommand{\pd}{\partial}
\newcommand{\D}{\mathrm{d}}
\newcommand{\la}{\langle}
\newcommand{\ra}{\rangle}

\newcommand{\csch}{\operatorname{csch}}

\newcommand{\dD}{\mathcal{D}}
\newcommand{\jump}[2]{[#1]_{#2}}

\newcommand{\bbU}{\mathbb{U}}

\newcommand{\sym}{\mathrm{sym}}

\newtheorem{theorem}{Theorem}[section]

\newtheorem{proposition}{Proposition}[section]
\newtheorem{corollary}{Corollary}[section]

\newtheorem{remark}{Remark}[section]

\newtheorem{example}{Example}[section]

\title[$\delta'$ leaky star optimization]
{An optimization problem for $\delta'$ leaky stars}

\author[Pavel Exner]{Pavel Exner\,$^{1,2}$}
\address{$^1$ Doppler Institute for Mathematical Physics and Applied Mathematics, Czech Technical University,  B\v rehov\'a 7, 11519 Prague, Czechia}
\address{$^2$ Department of Theoretical Physics, Nuclear Physics Institute, Academy of Sciences, Hlavn\'{\i} 130, 25068 \v{R}e\v{z} near Prague, Czechia}
\email{exner@ujf.cas.cz}

\begin{document}

\begin{abstract}
We consider singular Schr\"odinger operators with an attractive $\delta'$ interaction of strength $\beta$ supported by a star graph and prove that the ground-state eigenvalue is uniquely maximized by stars of the maximum symmetry. Two different proofs are presented, one for an even number of edges using a comparison to the $\delta$ interaction star, the other general based on Birman-Schwinger technique.
\end{abstract}


%

%

%
%
%
\maketitle

\section{Introduction} \label{s: intro}

Singular interactions are useful as a source of solvable models of quantum mechanical systems. This concerns both point interactions discussed in detail in \cite{AGHH} as well as those for which the interaction support is a curve or a graph of lower dimensionality, see e.g. \cite{EI01, EN03, BLL13}. Such singular Schr\"odinger operators offer, in particular, various spectral-geometry-type problems. Among them, there is the question about the geometry of a star graph which maximizes the ground state eigenvalue; in case of $\delta$ interactions it was shown that this happens for configurations of maximum symmetry \cite{EL18}. In contrast to $\delta$, the more singular interaction $\delta'$ is less explored. In \cite{PV22}, it was proved that the discrete spectrum of a $\delta'$ star is finite and the case of the simplest star, a broken line, was examined in detail; it was shown that the eigenvalues are monotonous with respect to the angle and their asymptotics was found in the situation when the angle between the two halflines tends to zero.

The aim of the present paper is to provide a $\delta'$ counterpart to the result of \cite{EL18}, namely to show that the maximizing configuration is again an equiangular star. We will do it in two different ways. For even number $N$ of the edges we use the fact that the star then divides the plane into a two-color set, in which case the inequality between the $\delta$ and $\delta'$ star operators proved in \cite{BEL14} relates the maximizing configuration of the latter with strength $\beta$ to that of the former with strength $4\beta^{-1}$. The second approach imposes no restriction on the parity of $N$; it uses the Birman-Schwinger technique. The corresponding claims and their proofs are presented in Sections~3 and 4, respectively. As a preliminary, we discuss in the next section the one-dimensional counterpart of our problem, optimization for a finite number of $\delta'$ interactions on a loop. For a strong enough $\delta'$ interaction, we also indicate here the configuration minimizing the highest eigenvalue.

\section{Preliminaries: $\delta'$ interactions on a loop}
\label{s:1Dloop}
\setcounter{equation}{0}

To begin with, let us consider a finite number $N\ge 2$ of $\delta'$ interactions of identical strength placed on a loop regarded as $\T=\R/\ell$ with some $\ell>0$; in view of the scaling properties we may assume without loss of generality that the loop length is $\ell=2\pi$. In other words, we consider the self-adjoint operator $H_{\beta,Y}$ on $\HH = L^2(I)$ with $I=(0,2\pi)$ associated with the quadratic form
\begin{equation} \label{loopform}
     q_{\beta,Y}[f] = \int_0^{2\pi} |f'(x)|^2\D x - \beta^{-1} \sum_{j=1}^N |[f](y_j)|^2, \quad [f](y_j):= f(y_j+)-f(y_j-),
\end{equation}
for given $\beta\ne 0$ and $Y=Y_N:=\{y_j\}_{j=1}^N$ such that $0<y_1<\cdots<y_N<2\pi$; the form domain consists of functions from $H^1(I\setminus Y)$ obeying periodic boundary conditions. The operator $H_{\beta,Y}$ acts as a negative Laplacian on functions from $H^2(I\setminus Y)$ satisfying the conditions
   \begin{align}
     & f'(y_j+) = f'(y_j-) =:f'(y_j)\,, \quad f(y_j+)-f(y_j-)= -\beta f'(y_j)\,, \quad j=1,\dots,N\,, \label{1Ddelta'} \\
     & f(2\pi-) = f(0+)\,, \quad f'(2\pi-) = f'(0+)\,, \nonumber 
   \end{align}
cf.~\cite[Chap.~II.3]{AGHH}. In contrast to this standard source we use the opposite sign convention for the coupling constant $\beta$, more convenient in case of attractive interactions which are our main interest. It is easy to check that its spectrum is purely discrete, $\sigma(H_{\beta,Y}) = \{\lambda_n(\beta,Y)\}_{n=1}^\infty$, conventionally arranged in the ascending order and accumulating at infinity. Let us mention some other spectral properties.
\begin{proposition} \label{1disc}
In each interval $(-\infty,0)$ and $\big(n^2,(n+1)^2\big),\; n\in\N_0$, there is at most $N$ eigenvalues of $H_{\beta,Y}$. We have $\lambda_1(\beta,Y)=0$ for any $\beta<0$, while for $\beta>0$ and $N\ge 2$ the negative spectrum is nonempty of total multiplicity at least $N-1$.
\end{proposition}
\begin{proof}
The first claim follows from general principles \cite[Corr.~to Thm.~8.19]{We80} because $H_{\beta,Y}$ and the Laplacian of the loop have a common symmetric restriction of deficiency indices $(N,N)$. By \eqref{loopform}, the operator $H_{\beta,Y}$ is positive in the repulsive case, $\beta<0$, and zero is an eigenvalue for any $\beta$ referring to the constant eigenfunction $f(x)=(2\pi)^{-1/2}$. On the other hand, using step functions constant in each interval $(y_j,y_{j+1})$ one can check that $\inf_{f\in H^1(I\setminus Y)} q_{\beta,Y}[f]<0$ for $\beta>0$. Moreover, among such trial functions one find $N-1$ linearly independent mutually and with the trivial one, so the last claim follows by minimax principle. This argument naturally fails for $N=1$, and indeed, in that case the negative spectrum may be empty; it is straightforward to check that this happens if and only if $\beta\ge 2\pi$.
\end{proof}

\begin{remark} \label{rem:nopreserve}
{\rm While the $\delta$ interaction counterpart $H^\delta_{\alpha,Y}$ to our operator, associated with the form \eqref{deltaform} below shares the property of the `usual' Schr\"odinger operators having a simple ground state with positive eigenfunction, the situation with  $H_{\beta,Y}$ is different. The positivity is violated already for $N=1$ where the (unnormalized) ground state eigenfunction for $\beta\in(0,2\pi)$ equals $f(x)= \sinh\kappa(x-\pi)$, where $\kappa$ is the solution to $\frac{2}{\kappa}\tanh\kappa\pi = \beta$. More generally, the existence of simple ground state with positive eigenfunction requires $\exp(-tH_{\beta,Y})$ to be positivity preserving to any $t>0$ which, by Beurling-Deny criterion \cite[Thm.~XIII.50]{RS78}, is equivalent the inequality $q[f]\ge q[|f|]$ for all $f\in H^1(I\setminus Y)$. This is not the case, though, since one can find easily an $f$ for which $[|f|](y_j) < [f](y_j)$, and indeed, the ground state may be degenerated which happens for any odd $N$, cf.~Example~\ref{ex:threepoint} and Remark~\ref{rem:spectrum}.}
\end{remark}

We are interested in the configurations $Y$ which optimize the ground state, more specifically, looking for its maximimizer. Note that the \emph{minimizer} is trivial for $\beta<0$, where $\lambda_1(\beta,Y)=0$ holds for any $Y$, and it does not exist for $\beta>0$.
\begin{proposition} \label{1nomin}
Given $\beta>0$, to any $c<0$ one can find an $Y$ such that $\lambda_1(\beta,Y)<c$.
\end{proposition}
\begin{proof}
It is convenient to extend $Y$ periodically; then one can obviously replace $(0,2\pi)$ by any interval on length $2\pi$ the endpoints of which have an empty intersection with $Y$. It is also clear from \eqref{loopform} that we may restrict ourselves to the case $N=2$. We then take $I=(-\pi,\pi)$ and place the interaction to $y_j=(-1)^j\eta,\: j=1,2$. Choosing $f(x)=\cosh x$ for $|x|<\eta$ and zero otherwise, we have $\|f\|^2= \eta+\frac12\sinh2\eta$ a and $\|f'\|^2= -\eta+\frac12\sinh2\eta$, hence the Rayleigh quotient for this function is
$$
\frac{q_{\beta,Y}[f]}{\|f\|^2} = \frac{\|f'\|^2 - 2\beta^{-1}(|f[-\eta]|^2 + |f[\eta]|^2)}{\|f\|^2} = \frac{-\eta+\frac12\sinh2\eta - 2\beta^{-1}\cosh2\eta}{\eta+\frac12\sinh2\eta}.
$$
Taking now Taylor expansion of the right-hand side expression, we get $-\frac{\beta^{-1}}{\eta}\big(1+\OO(\eta^2)\big)$ which can be obviously made arbitrarily large negative by choosing $\eta$ small enough.
\end{proof}

Before proceeding further, let us recall the the $\delta$ interaction counterpart of the current problem, the operator $H^\delta_{\alpha,Y}$ on $L^2(I)$ for which conditions \eqref{1Ddelta'} are replaced~by
$$
f(y_j+) = f(y_j-) =:f(y_j)\,, \quad f'(y_j+)-f'(y_j-)= -\alpha f(y_j)\,, \quad j=1,\dots,N\,;
$$
it is associated with the quadratic form
\begin{equation} \label{deltaform}
     q^\delta_{\alpha,Y}:\; q^\delta_{\alpha,Y}[f] = \|f'\|^2 - \alpha \sum_{j=1}^N |f(y_j)|^2
\end{equation}
defined on $H^1(I)$ with periodic boundary conditions. The resolvent of $H^\delta_{\alpha,Y}$ can be expressed by means of Krein's formula,
   \begin{align*}
     & (H^\delta_{\alpha,Y}-k^2)^{-1} = G_k + \sum_{j,j'=1}^N [\Gamma^\delta_{\alpha,Y}(k)]_{jj'}^{-1} \big( \overline{G_k(\cdot - y_{j'})}, \cdot \big) G_k(\cdot - y_j)\,, \\[.3em] 
     & \Gamma^\delta_{\alpha,Y}(k) = - [-\alpha^{-1}\delta_{jj'} + G_k(y_j-y_{j'})]_{j,j'=1}^N\,, 
   \end{align*}
where $G_k=(-\Delta-k^2)^{-1}$ is the free resolvent, an integral operator with the kernel
\begin{equation} \label{1Dkernel}
     G_k(x-y) = \frac{\cos(k(\pi-|x-y|)}{2k\sin\pi k}
\end{equation}
defined for $\mathrm{Im}\,k>0$ as well as for real $k\not\in\Z$; the eigenvalues of $H^\delta_{\alpha,Y}$ are the squared roots of the equation $\det \Gamma^\delta_{\alpha,Y}(k)=0$. For attractive $\delta$ interactions the lowest one is maximized by the configuration of maximum symmetry, for the repulsive ones we need to impose a restriction on the coupling constant.
\begin{proposition} \label{prop:1repuls}
Let $\tilde{Y}:= \{ \frac{\pi}{N}(2j-1)\}_{j=1}^N$, then to any $N\ge 2$ there is an $\alpha_N<0$ such that for any $Y$ with $\#Y=N$ we have $\lambda^\delta_1(\alpha,Y) \le \lambda^\delta_1(\alpha,\tilde{Y})$ provided $\alpha> \alpha_N$, the inequality being saturated if and only if $Y$ is congruent with $\tilde{Y}$.
\end{proposition}
\begin{proof}
For $\alpha<0$ the claim was proved in Theorem~2.1 of \cite{Ex19}; for small enough positive $\alpha$ see Theorem~5.1 there.
\end{proof}
\begin{remark} \label{rem:small}
{\rm (a) The coupling constant smallness of repulsive $\delta$ interactions is not necessary for symmetric configuration to maximize the ground state eigenvalue, cf.~\cite[Theorem~5.2]{Ex19}. The question whether the claim is true generally remains open. One can prove at least that $\tilde{Y}$ is a \emph{local} maximizer for any $\alpha<0$, see the appendix. \\ (b) The value of $\lambda^\delta_1(\alpha,\tilde{Y})$ is not difficult to find. The eigenfunction is hyperbolic or trigonometric depending on the sign of $\alpha$, $\,\frac{2\pi}{N}$-periodic, with $N$ points where the derivative vanishes, and the corresponding eigenvalue comes from the solution of the transcendent equations $2\kappa\tanh\frac{\pi\kappa}{N}=\alpha$ and $2k\tan\frac{\pi k}{N}=\alpha$ for $\pm\alpha>0$, respectively. This shows that the eigenvalue behaves as $\frac{N}{2\pi}\alpha + \OO(\alpha^2)$ around $\alpha=0$, as $-\frac14\alpha^2 +\OO(\alpha)$ when $\alpha\to\infty$, and it tends to the Dirichlet value $\frac14 N^2$ as $\alpha\to-\infty$.}
\end{remark}

After this detour let us return to the $\delta'$ interactions and show that for some values of the coupling constant we can bound the operator $H_{\beta,Y}$ from above by its $\delta$ interaction family counterpart.
\begin{proposition} \label{prop:1bound}
Let $\chi$ be the chromatic number of the partition of $I$ referring to $Y$, two for an even $N$ and three for an odd one. Suppose that $\beta\le \frac{4}{\alpha}\sin^2\frac{\pi}{\chi}$, than there is a unitary operator $U$ such that $U^{-1} H_{\beta,Y}U \le H^\delta_{\alpha,Y}$.
\end{proposition}
\begin{proof}
The argument is analogous to the one used in \cite[Thm.~3.6]{BEL14} for attractive $\delta'$ interactions on a Lipschitz partition of a plane; let us sketch it. Consider complex numbers of modulus one, $z_1=1$ and $z_2=-1$ if $N$ is even and $z_{2,3}=\frac12\big(-1\pm\sqrt{3}\big)$ in $N$ is odd; for any $i\ne j$ we obviously have $|z_i-z_j|^2=4\sin^2\frac{\pi}{\chi}$, four if $N$ is even and three if it is odd. We put $\alpha_{\cZ}:= \frac{1}{\beta}\sin^2\frac{\pi}{\chi}$ so by assumption $\alpha_{\cZ}\le\alpha$. Denoting $I_j:=(y_j,y_{j+1})$, we define the operator $U_{\cZ}$ by
\begin{equation} \label{1unit}
(U_{\cZ}f)(x) = z_jf_j(z) \quad\text{for}\;\; x\in I_f,\;\;\text{where} \\;\;\; f_j=f\big|_{I_j}.
\end{equation}
The operator is obviously unitary and allows us to define the form $\breve{q}_\beta[f]:= q_\beta[U_{\cZ} f]$ with $\mathrm{dom}\,\breve{q}_\beta = \mathrm{dom}\,q_\beta$, which satisfies $\breve{q}_\beta \le q^\delta_\alpha$. Indeed, we have
\begin{align*}
\breve{q}_\beta[f] & = \sum_{j=1}^N \|z_j\nabla f_j\|^2_{L^2(I_j)} + \frac{1}{\beta} \sum_{j=1}^N |z_j-z_{j+1}|^2 |f(y_j)|^2 \\
& = \|\nabla f\|^2_{L^2(I)} + \alpha_{\cZ} \sum_{j=1}^N |f(y_j)|^2 = q^\delta_{\alpha_{\cZ}}[f] \le q^\delta_{\alpha}[f].
\end{align*}
for any $f\in \mathrm{dom}\,q_{\alpha_{\cZ}} = H^1(I) \subset \bigoplus_{j=1}^N H^1(I_j) = \mathrm{dom}\,\breve{q}_\beta$.
\end{proof}

\begin{corollary} \label{cor:1even}
Suppose that $N$ is even and $\beta>0$, then the ground state is simple and we have $\lambda_1(\beta,Y) \le \lambda_1(\beta,\tilde{Y})$, the inequality being saturated if and only if $Y$ is congruent with $\tilde{Y}$. The function $\lambda_1(\cdot,\tilde{Y})$ is increasing with $\lambda_1(\beta,\tilde{Y})=-\frac{4}{\beta^2}+\OO(\beta^{-1})$ as $\beta\to 0+$ and $\lambda_1(\beta,\tilde{Y})=-\frac{2N}{\beta}+\OO(\beta^{-2})$ as $\beta\to\infty$.
\end{corollary}
\begin{proof}
Since $\chi=2$ in this case, $\lambda_1(\beta,Y) \le \lambda^\delta_1(4\beta^{-1},Y)$ holds by Proposition~\ref{prop:1bound}, thus Proposition~\ref{prop:1repuls} implies that $\lambda_1(\beta,Y) \le \lambda^\delta_1(4\beta^{-1},\tilde{Y})$ for any $Y$. Let $f_\mathrm{max}$ be the eigenfunction of $H^\delta_{4\beta^{-1},\tilde{Y}}$ corresponding to $\lambda^\delta_1(4\beta^{-1},\tilde{Y})$, that is, satisfying $-f_\mathrm{max}''(x) =\lambda^\delta_1(4\beta^{-1},\tilde{Y})f_\mathrm{max}(x)$ for $x\in I\setminus\tilde{Y}$ and $f_\mathrm{max}'(\tilde{y}_j\pm) = \pm 2\beta^{-1}f_\mathrm{max}(\tilde{y}_j)$ with $\tilde{y}_j=\frac{\pi}{N}(2j-1)$. The function $\breve{f}_\mathrm{max}:\, \breve{f}_\mathrm{max}(x) = (-1)^j f_\mathrm{max}(x)$ for $x\in I_j$ satisfies the same equation; its derivative is continuous while the value switches sign at the points $\tilde{y}_j$ satisfying the conditions
$$
\breve{f}_\mathrm{max}(\tilde{y}_j+) - \breve{f}_\mathrm{max}(\tilde{y}_j-) = \beta \breve{f}'_\mathrm{max}(\tilde{y}_j),
$$
in other words, it is nothing but an eigenfunction of $H(\beta,\tilde{Y})$ corresponding to the eigenvalue $\lambda^\delta_1(4\beta^{-1},\tilde{Y})$. Among the eigenvalues of $H(\beta,\tilde{Y})$, this one is the lowest, $\lambda_1(\beta,\tilde{Y})$. Suppose that this is not the case. Such a ground state, possibly degenerate, would be between the points of $\tilde{Y}$ a hyperbolic function of $\kappa x$ with $\kappa> \sqrt{-\lambda^\delta_1(4\beta^{-1},\tilde{Y})}$ and in view of the symmetry of $\tilde{Y}$ it would be a linear combination of functions either symmetric or antisymmetric with respect to all the points of $\tilde{Y}$. However, the former will not satisfy conditions \eqref{1Ddelta'} while the latter would by the inverse unitary transformation produce an eigenfunction of $H^\delta_{\alpha,\tilde{Y}}$ corresponding to an eigenvalue below $\lambda^\delta_1(4\beta^{-1},\tilde{Y})$ which is a contradiction. From the same reason the ground state of $H(\beta,\tilde{Y})$ cannot be degenerate. The behavior of the ground state eigenvalue follows easily from the relation $\kappa\tanh\frac{\pi\kappa}{N}=\frac{2}{\beta}$ determining it, which follows from the equation mentioned in Remark~\ref{rem:small}(b) with the substitution $\alpha=4\beta^{-1}$
\end{proof}

The ground-state eigenfunctions of the optimizing configuration can be easily found; as an example we plot the one for $N=4$.
	\begin{figure}[htb!]
		\centering
		\includegraphics[scale=0.4]{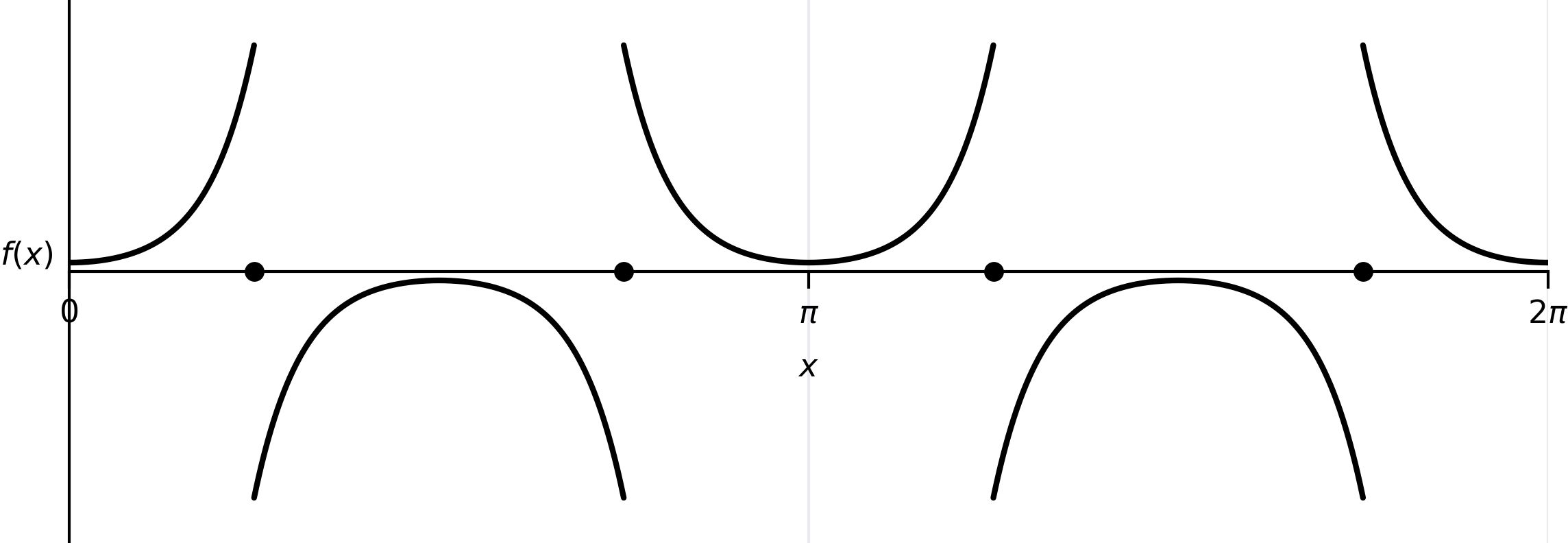}
		\caption{The eigenfunction corresponding to $\lambda_1(\beta,\tilde{Y}_4)$ for $\beta=\frac25$, the bullets indicate the $\delta'$ interactions positions.}
	\end{figure}
\begin{remark} \label{rem:attr_only}
{\rm (a) $H(\beta,\tilde{Y})$ has for each $\beta\in\big(0,\frac{2\pi}{N}\big)$ an eigenvalue with the eigenfunction locally antisymmetric with respect to all the points of $\tilde{Y}$; it consists of the segments of the form $\sinh\kappa\big(x-\frac12(\tilde{y}_j+\tilde{y}_{j+1})\big)$. The corresponding eigenvalue $-\kappa^2$ comes from solution to $\kappa\coth\frac{\pi\kappa}{N}=\frac{2}{\beta}$. Comparing this to $\kappa\tanh\frac{\pi\kappa}{N}=\frac{2}{\beta}$, we see this eigenvalue is strictly larger that $\lambda_1(\beta,\tilde{Y})$. It is also not difficult to find the eigenfunction refererring to $\lambda_1(\beta,\tilde{Y})$; without normalization, it equals $(-1)^j \cosh\kappa\big(x-\frac12(\tilde{y}_j+\tilde{y}_{j+1})\big)$ for $x\in(\tilde{y}_j,\tilde{y}_{j+1})$.\\
(b) While the claim of Proposition~\ref{prop:1bound} holds for both coupling constant signs, the above corollary does not extend to the repulsive case even if we assume $\beta<\frac{4}{\alpha_N}$ only. The reason is that the ground state eigenvalue of $H_{\beta,\tilde{Y}}$ does not coincide in this case with that of $H^\delta_{\alpha,\tilde{Y}}$. Indeed, the latter is positive, however, the former is zero (for any $Y$) as pointed out in Proposition~\ref{1disc}.}
\end{remark}

The above argument fails if $N$ is odd, in particular, because in analogy with \cite[Cor.~3.12]{BEL14} the unitary operator of Proposition~\ref{prop:1bound} does not exist for $\beta> \frac{4}{\alpha}\sin^2\frac{\pi}{\chi}$ and $\chi>2$ now. Instead one can use the Krein's formula expression of the resolvent, which in analogy with \cite[Thm.~II.3.1]{AGHH} now reads
   \begin{align*}
     & (H_{\beta,Y}+\kappa^2)^{-1} = G_\kappa + \sum_{j,j'=1}^N [\Gamma_{\beta,Y}(\kappa)]_{jj'}^{-1} \big( \overline{\tilde{\GG}_\kappa(\cdot - y_{j'})}, \cdot \big) \tilde{\GG}_\kappa(\cdot - y_j), \\[.3em] 
     & \Gamma_{\beta,Y}(\kappa) = [-\beta^{-1}\delta_{jj'} + \kappa^2 G_\kappa(y_j-y_{j'})]_{j,j'=1}^N\,, \label{delta'krein2} \\[1em]
     & \tilde{\GG}_\kappa(x-y) = \mathrm{sgn}(x-y)\kappa G_\kappa(x-y). 
   \end{align*}
where $G_\kappa$ is again the free resolvent given by \eqref{1Dkernel}, now for $k=i\kappa$,
\begin{equation} \label{1Dkernel_attr}
     G_\kappa(x-y) = \frac{\cosh(\kappa(\pi-|x-y|)}{2\kappa\sinh\pi\kappa}.
\end{equation}
We state Krein's formula in this form because in view of Remark~\ref{rem:attr_only} it makes sense to consider attractive $\delta'$ interactions only, $\beta>0$. In that case we have $\lambda_1(\beta,Y)=-\kappa^2$, with $\kappa$ being the smallest positive root of
\begin{equation} \label{1secular}
\det\,\Gamma_{\beta,Y}(\kappa) = 0,
\end{equation}
or alternatively the solution to
$$ 
\inf\sigma\big(\Gamma_{\infty,Y}(\kappa)\big) = \frac{1}{\beta}.
$$ 
More generally, the negative eigenvalues $\lambda_j(\beta,Y)$ are equal to $-\kappa_j^2$, where $\kappa_j$ is the solution to
$$ 
\mu_j(\kappa) = \frac{1}{\beta},
$$ 
where $\mu_j(\kappa)$ is the $j$th eigenvalue of $\Gamma_{\infty,Y}(\kappa)$. These solutions are unique, since  the matrix-valued function $\Gamma_{\infty,Y}(\cdot)$ is monotonously increasing as the spectral parameter $\kappa$ runs through $(0,\infty)$, cf. \cite{Kr46}, with $\lim_{\kappa\to\infty}\mu_j(\kappa)=\infty$; note that for large $\kappa$ the matrix $\Gamma_{\infty,Y}(\kappa)$ is dominated by its diagonal the entries of which are $\frac{\kappa}{2}\coth\pi\kappa$.

As we know from Proposition~\ref{1disc}, the $N$th negative eigenvalue may or may not exist depending on the value of $\beta$. Indeed, the largest eigenvalue of $\Gamma_{\infty,Y}(\kappa)$ is simple by Perron-Frobenius and in the limit $\kappa\to 0$ all the entries of the matrix become $(2\pi)^{-1}$, hence $\lim_{\kappa\to 0} \mu_N(\kappa)= \frac{N}{2\pi}$ meaning that $\lambda_N(\beta,Y)<0$ holds if and only if $\beta < \frac{2\pi}{N}$. On the other hand, in accordance with Proposition~\ref{1disc} we have $\lambda_j(\beta,Y)<0$ for $j=1,\dots,N-1$ and any coupling $\beta>0$ because $\lim_{\kappa\to 0} \mu_j(\kappa)= 0$.

\begin{example} \label{ex:threepoint}
{\rm Consider the case $N=3$ and denote $d_{ij}=y_i-y_j$ cyclically; then we have $d_{31}=2\pi-(d_{12}+d_{23})$. For simplicity, we put $g_{ij}=G_\kappa(d_{ij})$; the Perron-Frobenius eigenvalue is squeezed between $\max_i \sum_j g_{ij}$ and $\min_i \sum_j g_{ij}$. The function $f(d):=\cosh\kappa(\pi-d)$ is strictly convex, hence $f(d_{12}+d_{23}+d_{31})\ge 3f\big(\frac23\pi\big)$, the inequality being saturated \emph{iff} $d_{12}=d_{23}=d_{31}=\frac23\pi$. Consequently, the sum $\frac13\big(3g_{ii}+2(d_{12}+d_{23}+d_{31})\big)$ appearing in both the indicated bounds is strictly minimized by the equidistant configuration for which $\mu_3(\kappa)=a+2b$, where $a=\frac{\kappa}{2}\coth\pi\kappa$ and $b=\frac{\kappa\cosh(\pi\kappa/3)}{2\sinh\pi\kappa}$; this gives $\mu_3(\kappa)=\frac{\kappa}{2}\coth\frac{\pi\kappa}{3}$ corresponding to then eigenvector $3^{-1/2}(1,1,1)$, and therefore the highest negative eigenvalue of $H(\beta,Y_3)$ is \emph{strictly minimized} by the one with the maximum symmetry, $\lambda_3(\beta,Y_3) > \lambda_3(\beta,\tilde{Y}_3)$, unless $Y_3$ is congruent with $\tilde{Y}_3$.

Furthermore, the two remaining eigenvalues of $\Gamma_{\beta,\tilde{Y}_3}(\kappa)$ coincide mutually,
$$
\mu_1(\kappa)=\mu_2(\kappa)=a-b = \frac{\kappa \sinh\frac{2\pi\kappa}{3}}{1+2\cosh\frac{2\pi\kappa}{3}}
$$
with the eigenvectors $3^{-1/2}(1,\ee^{\pm 2\pi i/3},\ee^{\pm 4\pi i/3})$, and as a result, the ground state eigenvalue of $H(\beta,\tilde{Y}_3)$ has \emph{multiplicity two}. In the weak coupling case, $\beta\to\infty$, we have $\lambda_1(\beta,\tilde{Y}_3) = \lambda_2(\beta,\tilde{Y}_3) = -\frac{3}{2\pi\beta} + \OO(\beta^{-2})$, while for strong coupling, $\beta\to 0+$, all the three eigenvalues behave as $\lambda_j(\beta,\tilde{Y})=-\frac{4}{\beta^2}+\OO(\beta^{-1})$, the differences concerning the error term only.}
\end{example}

\begin{theorem} \label{thm:1lambdaN}
For any $\beta\in\big(0,\frac{2\pi}{N}\big)$, the eigenvalue $\lambda_N(\beta,Y)$ of $H(\beta,Y)$ is mini\-mized, uniquely up to shifts, by the symmetric configuration, $Y=\tilde{Y}$, the eigenvalue $-\kappa^2$ being determined by the condition $\kappa \coth\!\frac{\kappa\pi}{N} = \frac{2}{\beta}$.
\end{theorem}
\begin{proof}
Denote again $d_{jk} = \min\bigl\{|y_j-y_k|,\, 2\pi-|y_j-y_k|\bigr\}$ the geodesic distances on the loop. The largest eigenvalue of the matrix $\Gamma_{\infty,Y}(\kappa)$ equals
$$
\mu_\mathrm{\max}(\kappa) = \max_{\|c\|_2=1} \sum_{j,k=1}^{N} c_jc_k\,\kappa^2G_\kappa(d_{jk}).
$$
and, again by Perron-Frobenius, the corresponding eigenvector $c\in\C^N$ can be chosen positive. The function $\kappa^2G_\kappa(\cdot)$ is decreasing, hence an increase of pairwise distances decreases the matrix entries, and strictly convex. Next we denote $\delta_j:=y_{j+1}-y_j$, cyclically, the successive gaps between the points of $Y$; they are positive and $\sum_{j=1}^N \delta_j=2\pi$. The convexity allows us to apply a Fejes T\'oth type majorization argument, replacing neighboring gaps $(\delta_j,\delta_{j+1})$ by their arithmetic mean, $(\delta_j,\delta_{j+1}) \mapsto \left( \frac12(\delta_j+\delta_{j+1}), \frac12(\delta_j+\delta_{j+1})\right)$ which makes the Perron-Frobenius eigenvalue smaller whenever $\delta_{j+1}\ne \delta_j$. Iterating this averaging procedure, we arrive at the conclusion that the minimum of the form is reached when all the gaps are equal, $\delta_j=\frac{2\pi}{N},\, j=1,\dots,N$, i.e. if $Y$ is congruent with $\tilde{Y}$.

One can also write the equation determining the extremal value of $\lambda_N(\beta,Y)$. For a symmetric $Y$ the matrix $\Gamma_{\infty,Y}(\kappa)$ is circulant \cite{Da79} and its largest eigenvalue equals the sum of any row or column,
$$
\min_Y \mu_\mathrm{\max}(\kappa) = \sum_{m=0}^{N-1} \frac{\kappa}{2\sinh(\pi\kappa)} \cosh\!\left( \kappa \Bigl( \pi-\frac{2\pi}{N}\min(m,N-m) \Bigr) \right) = \frac{\kappa}{2}\coth\!\left(\frac{\kappa\pi}{N}\right),
$$
which in view of \eqref{1secular} gives rise to the spectral condition $\kappa \coth\!\frac{\kappa\pi}{N} = \frac{2}{\beta}$ which we have already encountered in Remark~\ref{rem:attr_only}.
\end{proof}

The eigenfunction corresponding the symmetric configuration minimizing $\lambda_5(\beta,Y)$ is again easy to find; we plot an example with $N=5$.
	\begin{figure}[htb!]
		\centering
		\includegraphics[scale=0.4]{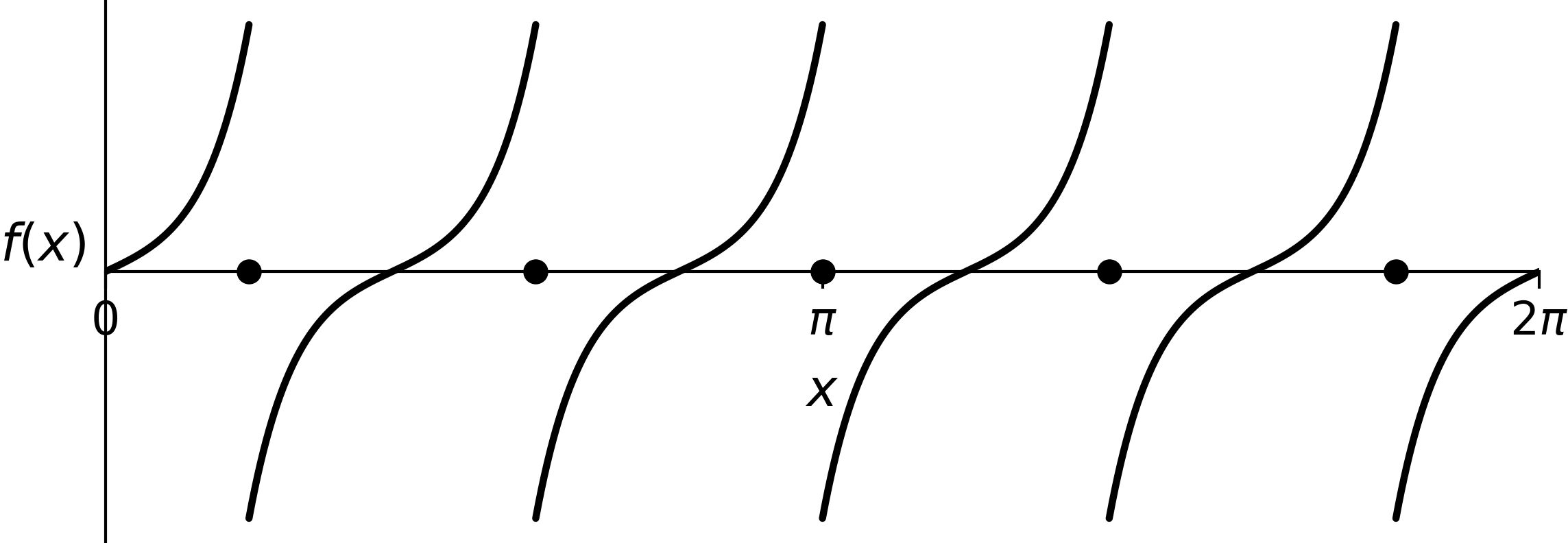}
		\caption{The eigenfunction corresponding to $\lambda_5(\beta,\tilde{Y})$ for $\beta=\frac25$.}
	\end{figure}

Let us now return to the question of our main concern, namely maximization of the ground state. Here, too, the symmetry of $Y$ plays the fundamental role.
\begin{theorem} \label{thm:1lambdaN}
The eigenvalue $\lambda_1(\beta,Y)$ of $H(\beta,Y)$ is maxi\-mized, uniquely up to shifts, by the symmetric configuration, $Y=\tilde{Y}$. The multiplicity of $\lambda_1(\beta,\tilde{Y})$ is one for $N$ even and two for $N$ odd.
\end{theorem}
\begin{proof}
We can no longer employ Perron-Frobenius, instead we rephrase the question in terms of the inverse matrix. The function \eqref{1Dkernel_attr} is even and symmetric with respect to $x=\pi$ and extends periodically to the real line; it satisfies $-G_\kappa''+\kappa^2G_\kappa=0$ for $x\neq 2\pi k,\, k\in\N$, and $G_\kappa'(0\pm)=\mp\tfrac12$. Knowing the spectrum of the free operator on the loop, one write it also as $G_\kappa(x)=\frac1{2\pi}\sum_{n\in\Z}\frac{e^{inx}}{n^2+\kappa^2}$, which means that
$$
\la c,\Gamma_{\infty,Y}(\kappa)c\ra =\frac{\kappa^2}{2\pi}\sum_{n\in\Z}\frac{\big|\sum_{j}c_je^{iny_j}\big|^2}{n^2+\kappa^2}>0
$$
holds for any nonzero $c\in\C^N$ since for distinct $y_j$ the trigonometric sums cannot all vanish. The matrix $\Gamma_{\infty,Y}(\kappa)$ is thus positive and so is its inverse, and we have
\begin{equation} \label{invtrick}
\mu_\mathrm{\min}(\kappa)=\big(\nu_\mathrm{\max}(\kappa)\big)^{-1},
\end{equation}
where $\nu_\mathrm{\max}(\kappa)$ is the maximal eigenvalue of $\Gamma_{\infty,Y}(\kappa)^{-1}$. To use relation \eqref{invtrick}, one has derive the inverse matrix which appears a particular Jacobi matrix. To find it, note that for any given $c\in\C^N$, the function $\psi=\sum_{j'}c_{j'}\,G_\kappa(\,\cdot\,-y_{j'})$ solves the equation $-\psi''+\kappa^2\psi=0$ on each interval $(y_j,y_{j+1})$, is continuous there and has derivative jumps at the endpoints, $\psi'(y_j-)-\psi'(y_j+)=c_j$. Its values at the points $y_j$, expressed in terms of $\Gamma_{\infty,Y}(\kappa)$, are $v_j:=\psi(y_j)=\kappa^{-2}(\Gamma_{\infty,Y}(\kappa)c)_j$.

Now we need to express the vector $c$ corresponding to given $v$. To this aim, we use the explicit form of the function $\psi$ as a solution to $-\psi''+\kappa^2\psi=0$: on
an interval of length $\ell$ with end values $a,b$ one has
$$
\psi(x)=\frac{a\sinh\kappa(\ell-x)+b\sinh\kappa x}{\sinh\kappa\ell}
$$
so computing the jumps $c_j$ of $\psi'$ from the values $v_k$ gives $c=M_Yv$ with
$$ 
(M_Y)_{jj}:=\kappa\big(\coth\kappa\delta_{j-1}+\coth\kappa\delta_j\big),
\quad
(M_Y)_{j,j+1}:=(M_Y)_{j+1,j}=-\frac{\kappa}{\sinh\kappa\delta_j}
$$ 
where $\delta_j=y_{j+1}-y_j$ are the distances of the neighboring points introduced in the previous proof and the indices are again taken modulo $N$. The sought inverse matrix is thus $\Gamma_{\infty,Y}(\kappa)^{-1}= \kappa^{-2}M_Y$. Using further the identities $\coth t-\frac1{\sinh t}=\tanh\frac t2$ and $\coth t+\frac1{\sinh t}=\coth\frac t2$, one can write the quadratic form of the matrix $M_Y$ as a sum over the intervals between the neighboring points,
\begin{equation}\label{qformM_Y}
\la v,M_Yv\ra=\frac{\kappa}{2}\sum_{j=1}^N\Big[\tanh\!\Big(\frac{\kappa\delta_j}{2}\Big)\,|v_j+v_{j+1}|^2
+\coth\!\Big(\frac{\kappa\delta_j}{2}\Big)\,|v_j-v_{j+1}|^2\Big].
\end{equation}
This can be in view of \eqref{invtrick} used to estimate the eigenvalue in question,
\begin{equation} \label{invtrick2}
\mu_\mathrm{\min}(\kappa)= \frac{\kappa^2}{\lambda_{\max}(M_Y)} \;\le\;\frac{\kappa^2}{\la v,M_Yv\ra}
\end{equation}
which holds for any unit vector $v\in\C^N$. This is the key point: we have to find a vector which minimizes $\la v,M_Yv\ra$ and show that it corresponds to $Y=\tilde{Y}$.

We note that both $M_{\tilde{Y}}$ and $\Gamma_{\infty,\tilde{Y}}(\kappa)$ are circulant, and as such they have common eigenvectors the components are complex roots of unity, namely $v^{(k)}_j=N^{-1/2}e^{2\pi i kj/N}$, $k=0,\dots,N-1$; the corresponding eigenvalues are easily found to be
$$ 
\Lambda_k(M_{\tilde{Y}})=\frac{2\kappa}{\sinh\kappa\delta}\Big[\cosh\kappa\delta-\cos\frac{2\pi k}{N}\Big]
$$ 
and
\begin{equation} \label{tilded_ev}
\Lambda_k(\Gamma_{\infty,\tilde{Y}}(\kappa))=\frac{\kappa^2}{\lambda_k(M_{\tilde{Y}})}
=\frac{\kappa}{2}\,\frac{\sinh(2\pi\kappa/N)}{\cosh(2\pi\kappa/N)-\cos(2\pi k/N)},
\end{equation}
where $\delta=2\pi/N$ is the common value of the distance $\delta_j$. The eigenvalues \eqref{tilded_ev} differ from the naturally ordered ones, $\mu_j(\kappa),\,j=1,\dots,N$, we are interested, in by numeration only; we have $\mu_j=\Lambda_{j-1}$ for $N$ even and $\mu_j=\mu_{j+1}=\Lambda_{\lfloor{(N\pm j)/2\rfloor}}$ for $N$ odd. In particular, the smallest eigenvalue, $\mu_1(\kappa)$, corresponds to $k$ closest to $N/2$,
\begin{equation}\label{eq:mu_min_opt}
\lambda_1(\beta,\tilde{Y})=
\begin{cases}
\dfrac{\kappa}{2}\tanh\dfrac{\pi\kappa}{N}, & N \text{ even},\\[4mm]
\dfrac{\kappa}{2}\,\dfrac{\sinh(2\pi\kappa/N)}{\cosh(2\pi\kappa/N)+\cos(\pi/N)}, & N \text{ odd},
\end{cases}
\end{equation}
being simple in the former case and of multiplicity two in the latter.

This allows us show by means of \eqref{invtrick2} that the optimal configuration is equidistant using the eigenvectors corresponding to \eqref{eq:mu_min_opt}. Consider first $N$ even and take $v_j:=(-1)^j\,N^{-1/2}$, then $v_j+v_{j+1}=0$ and $|v_j-v_{j+1}|^2=\frac{4}{N}$, hence \eqref{qformM_Y} gives
$$
\la v,M_Yv\ra=\frac{2\kappa}{N}\sum_{j=1}^N\coth\frac{\kappa\delta_j}{2}\;\ge\;2\kappa\coth\frac{\pi\kappa}{N},
$$
by Jensen's inequality, since $\coth$ is strictly convex on $(0,\infty)$ and $\frac1N\sum_j\delta_j=2\pi/N$, the inequality being saturated if and only if all the $\delta_j$ are equal. This yields
$$
\mu_{\min}(\kappa)\le\frac{\kappa^2}{2\kappa\coth(\pi\kappa/N)}=\frac{\kappa}{2}\tanh\frac{\pi\kappa}{N},
$$
showing that $\mu_{\min}(\kappa)$ is uniquely maximized by configurations $Y$ congruent with $\tilde{Y}$.

The argument for $N$ odd is similar, just slightly more complicated. One uses now $v_j=N^{-1/2}\ee^{i\theta j}$ with $\theta=\pi-\pi/N$ (note that this does not violate the periodicity since $e^{i\theta N}=e^{i\pi(N-1)}=1$ holds for odd $N$); then we have
$$
|v_j+v_{j+1}|^2=\frac2N\Big(1-\cos\frac{\pi}{N}\Big),\quad
|v_j-v_{j+1}|^2=\frac2N\Big(1+\cos\frac{\pi}{N}\Big),
$$
and
$$
\la v,M_Yv\ra=\frac{\kappa}{N}\sum_{j=1}^N h(\delta_j) \quad \text{with }\;
h(\delta):=\Big(1-\cos\frac{\pi}{N}\Big)\tanh\frac{\kappa\delta}{2}+\Big(1+\cos\frac{\pi}{N}\Big)\coth\frac{\kappa\delta}{2},
$$
where $h$ is strictly convex: putting $t=\kappa\delta/2$, $a=1-\cos\frac\pi N$, $b=1+\cos\frac\pi N$, we find
$$
\frac{d^2}{dt^2}\big[a\tanh t+b\coth t\big]
=2\Big(b\,\frac{\cosh t}{\sinh^3 t}-a\,\frac{\sinh t}{\cosh^3 t}\Big)
=2\,\frac{\sinh t}{\cosh^3 t}\,\big(b\coth^4 t-a\big)>0,
$$
because $b>a$ and $\coth^4 t>1$. Jensen's inequality then yields $\sum_j h(\delta_j)\ge N h(2\pi/N)$, with equality valid if and only if all the $\delta_j=2\pi/N$; using $\tanh t+\coth t=2\coth 2t$ and $\coth t-\tanh t=2/\sinh 2t$ with $t=\pi\kappa/N$, we get
$$
\la v,M_Yv\ra\;\ge\;\frac{2\kappa\,\big[\cosh(2\pi\kappa/N)+\cos(\pi/N)\big]}{\sinh(2\pi\kappa/N)}.
$$
Consequently, $\mu_{\min}(\kappa)$ is bounded from above by the second expression in \eqref{eq:mu_min_opt} with the equality valid if and only if $\delta_1=\dots=\delta_N=2\pi/N$; this concludes the proof.
\end{proof}

\begin{example} \label{ex:odd_eigenf}
{\rm In case of $N$ odd, the ground-state eigenfunctions of $H(\beta,\tilde{Y})$ are more complicated. It is not possible now that all the boundary values at the discontinuities flip sign; what happens is that such an eigenfunction has zero derivative at one of the $\delta'$ sites which means that the operator does not `see' the interaction. As an example, consider $N=5$, where the non-normalized eigenfunction is given by
\begin{equation} \label{Nfive}
f(x)=
\begin{cases}
A\cosh\kappa x, & 0<x<\dfrac{2\pi}{5},\\[3mm]
-A\Big[\cosh\kappa\big(x-\tfrac{4\pi}{5}\big)
+c\,\cosh\kappa\big(x-\tfrac{2\pi}{5}\big)\Big],
& \dfrac{2\pi}{5}<x<\dfrac{4\pi}{5},\\[3mm]
2c\,A\cosh\dfrac{\pi\kappa}{5}\,\cosh\kappa(\pi-x),
& \dfrac{4\pi}{5}<x<\pi
\end{cases}
\end{equation}
together with the extension to $(\pi,2\pi)$ even with respect to $x=\pi$, plotted below, where $c=\frac{1+\sqrt{5}}{2}$ is the golden mean. Another eigenfunction is obtained by a shift on an odd of multiple of $\pi$, modulo $2\pi$; it is obviously linearly independent as it should be for an eigenvalue of multiplicity two.
}
\end{example}

	\begin{figure}[htb!]
		\centering
		\includegraphics[scale=0.4]{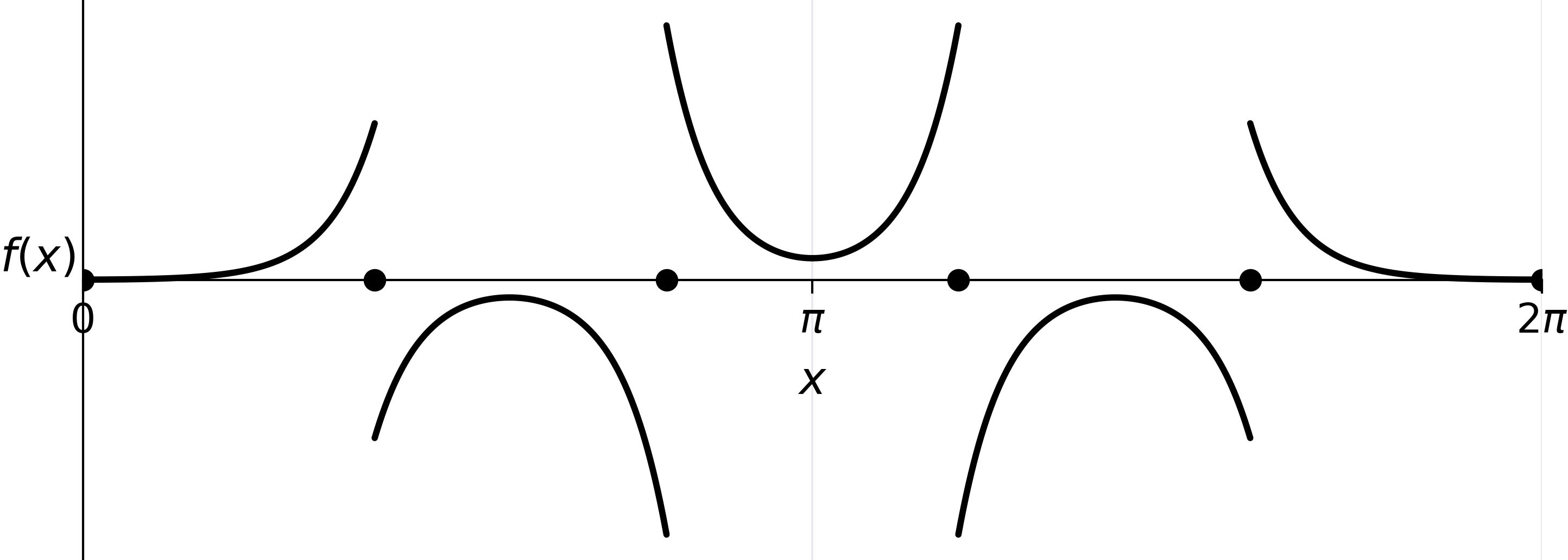}
		\caption{The eigenfunction corresponding to $\lambda_1(\beta,\tilde{Y}_5)$ for $\beta=\frac25$.}
	\end{figure}

\begin{remark} \label{rem:spectrum}
{\rm (a) As a byproduct, we have obtained the whole negative spectrum of $H(\beta,\tilde{Y})$, cf.~\eqref{tilded_ev}. It is simple for $N$ even, and of multiplicity two for $N$ odd, except possibly $\lambda_N(\beta,\tilde{Y})$ which is simple; recall that the eigenvalue $\mu_N(\kappa)=\Lambda_N(\Gamma_{\infty,\tilde{Y}}(\kappa))$ is always simple and we know from Remark~\ref{rem:attr_only}(a) that if gives rise to an eigenvalue of $H(\beta,\tilde{Y})$ if and only if $0<\beta<\frac{2\pi}{N}$. \\
(b) Note that, as for $N$ even, the ground-state eigenfunctions are discontinuous but have no zeros. This is no longer true for those referring to higher eigenvalues. In the above example, for instance, those corresponding to $\lambda_j(\beta,\Gamma),\, j=3,4,$ are given by \eqref{Nfive} with $c$ replaced by $-c^{-1}$ having thus two zeros.
}
\end{remark}

\section{A leaky $\delta'$ star}
\label{s:primestar}
\setcounter{equation}{0}

Let us now turn to our main topic, `leaky' $\delta'$ star graphs with the interaction supported by the set $\Gamma=\Gamma_N\subset \R^2$ which one can write using the polar coordinates in the plane as
\begin{equation} \label{star}
     \Gamma=\bigcup_{j=1}^N \Gamma_j, \quad \Gamma_j:=\{(r,\cos\theta_j):\, r\ge 0\},
\end{equation}
fully determined by the family of angles, $0\le \theta_1<\theta_2< \cdots < \lambda_N<2\pi$; the indices are understood modulo $N$ with $\theta_{N+1}:=\theta_1+2\pi$ and we consider only the nontrivial situations, $N\ge 2$. Alternatively, we can use that angles $\omega_j:= \theta_{j+1}-\theta_j$ (cyclically) satisfying $\sum_{j=1}^N \omega_j = 2\pi$ (our convention differs by factor two from that used in \cite{PV22}); the complement $\R^2\setminus\Gamma$ of $\Gamma$ is the union of the $N$ open sectors $\Omega_j=\{(r,\theta):\ \theta_j<\theta<\theta_{j+1}\}$. In particular, we employ the symbol $\tilde\Gamma$ for the symmetric (equiangular) star characterized by the angles $\tilde\omega_j= \frac{2\pi}{N},\, j=1,\dots,N$.

Following \cite{PV22}, we introduce the self-adjoint operator $H_{\beta,\Gamma}$ on $L^2(\R^2)$ by means of the associated quadratic form,
\begin{equation} \label{starform}
     q_{\beta,\Gamma}[f] = \int_{\R^2\setminus\Gamma} |\nabla f(x)|^2\D x - \beta^{-1} \int_\Gamma |[f](s)|^2\D s
\end{equation}
defined on $H^1(\R^2\setminus\Gamma)$, where the integrand of the second term contains the normal jump of the function over the star edges,
\begin{equation} \label{normjump}
     [f](s) := \lim_{t\to 0+} \big( f(s+tn(s)) - f(s-tn(s)) \big),
\end{equation}
where $n(s)$ is the unit normal field to $\Gamma$, defined and smooth away from the origin; for definiteness we suppose that it points from $\Omega_j$ to $\Omega_{j+1}$. By Sobolev trace theorem the map $f\mapsto [f]$ has unique bounded extension $H^1(\R^2\setminus\Gamma) \to L^2(\Gamma)$. The operator $H_{\beta,\Gamma}$ acts as $f\mapsto -\Delta f$ on functions $H^2(\R^2\setminus\Gamma)$ the Sobolev traces of which on $\Gamma$ satisfy the boundary conditions
\begin{equation} \label{delta'cond}
     \pd_n^+ f = \pd_n^- f =: \pd_n f \;\; \text{and} \;\; [f] = -\beta\pd_n f \;\; \text{with} \quad \pd_n^\pm f(s):= \lim_{t\to 0+} n(s)\cdot(\nabla f)(s\pm tn(s))\,;
\end{equation}
it can be understood as the rigorous version of the formal operator $\Delta +\beta^{-1}\delta'$.

In the repulsive case, $\beta<0$, the operator is easily seen to be positive with $\sigma(H_{\beta,\Gamma}) = \sigma_\mathrm{ess}(H_{\beta,\Gamma})=[0,\infty)$. We focus again on the attractive case and recall the basic properties of the operators $H_{\beta,\Gamma}$ derived in \cite{PV22}:
\begin{proposition} \label{prop:delta'basic}
$\sigma_\mathrm{ess}(H_{\beta,\Gamma})=[-4\beta^{-2},\infty)$ holds for any $\beta>0$. The discrete spectrum is nonempty and finite unless $N=2$ and $\theta_2-\theta_1=\pi$, i.e. unless $\Gamma$ is a straight line. With respect to the coupling parameter the spectrum is self-similar, because the operator $H_{\beta,\Gamma}$ is by scaling transformation unitarily equivalent to $\beta^{-2} H_{1,\Gamma}$.
\end{proposition}

In analogy with Proposition~\ref{1nomin} it makes no sense to ask for a minimizer.
\begin{proposition} \label{prop:nominimum}
For fixed $\beta>0,\, N\ge 2$, and any $c<0$ there is a star $\Gamma$ such that $\inf\sigma(H_{\beta,\Gamma})<c$.
\end{proposition}
\begin{proof}
It is clear from \eqref{starform} and the minimax principle that $\inf\sigma(H_{\beta,\hat\Gamma}) \le \inf\sigma(H_{\beta,\Gamma})$ holds if $\hat\Gamma \supset \Gamma$, hence it is enough to consider the case $N=2$ when $\Gamma$ is a broken line, two halflines meeting at a angle $\phi \in(0,\pi)$. By \cite[Thm.~2.1b]{PV22} the lowest eigenvalue of such an $H_{\beta,\Gamma}$ behaves asymptotically as
$$
\lambda_1(\beta,\Gamma) = 4\beta^{-2}\phi^{-2} + \OO(\phi^{-1})
$$
for $\phi\to 0+$, hence it can be made arbitrarily large negative what we set out to prove.
\end{proof}

If we look for a maximizing geometry, the parity of the vertex degree plays again role, the case of an even $N$ being easier:
\begin{theorem} \label{thm:Neven}
For even $N\ge 4$ and $\beta>0$, $\,\lambda_1(\beta,\Gamma)$ is uniquely maximized by $\tilde\Gamma$.
\end{theorem}
\begin{proof}
It makes no sense to consider the case $N=2$, because $\lambda_1(\beta,\Gamma)$ is by \cite[Thm.~2]{PV22} monotonous in the angle and disappears as $\Gamma$ approaches $\tilde\Gamma$ which is here a straight line. As in Corollary~\ref{cor:1even}, we use a comparison with the corresponding $\delta$ leaky graph with the Hamiltonian $H^\delta_{\alpha,\Gamma}$ on $L^2(\R^2)$ associated with the quadratic form,
$$ 
     q^\delta_{\alpha,\Gamma}[f] = \int_{\R^2} |\nabla f(x)|^2\D x - \alpha \int_\Gamma |f(s)|^2\D s
$$ 
defined on $H^1(\R^2)$. The chromatic number of $\Gamma$ is two by assumption, hence as long as $\beta\le \frac{4}{\alpha}$, there is a unitary operator $U$ such that $U^{-1} H_{\beta,\Gamma}U \le H^\delta_{\alpha,\Gamma}$. We choose $\alpha=\frac{4}{\beta}$ and use Theorem~1.1 of \cite{EL18} by which the ground state of $H^\delta_{4\beta^{-1},\Gamma}$ is maximized by that of $H^\delta_{4\beta^{-1},\tilde\Gamma}$, uniquely so, cf.~Proposition~\ref{prop:uniqdelta} below. Combining the two results, we arrive at the inequality
$$
\lambda_1(\beta,\Gamma) \le \lambda^\delta_1(4\beta^{-1},\tilde\Gamma).
$$
Denote by $\phi_\mathrm{max}$ the corresponding positive eigenfunction of $H^\delta_{4\beta^{-1},\tilde\Gamma}$. Since $N$ is even, $\tilde\Gamma_j$ and $\tilde\Gamma_{j+N/2}$ form a line for any $j$, and $\phi_\mathrm{max}$ is symmetric with respect to it, which in particular means that the normal derivative of $\phi_\mathrm{max}$ switches sign at the points of $\Gamma_j$.

The unitary operator $U$ we are using is known explicitly: denoting by $\Sigma_j$ the segment of the plane determined by the halflines $\tilde\Gamma_j$ and $\tilde\Gamma_{j+1}$, it acts as $(U\phi)(x) = (-1)^j \phi(x)$. Consequently, the function $U\phi_\mathrm{max}$ solves the equation $-\Delta\phi= \lambda^\delta_1(4\beta^{-1},\tilde\Gamma)\phi$ in $\R^2\setminus\tilde\Gamma$ and its normal derivative is continuous at each $\tilde\Gamma_j$, while its value has a jump there; denoting by $U\phi_\mathrm{max}(x\pm)$ the one-sided normal limits of this function at a point $x\in\tilde\Gamma_j$, we have
\begin{align*}
U\phi_\mathrm{max}(x+) - U\phi_\mathrm{max}(x-) & \: = 2\phi_\mathrm{max}(x) = -\frac{\beta}{2}\big(\phi'_\mathrm{max}(x+) - \phi'_\mathrm{max}(x-) \\ & \: = -\beta\phi'_\mathrm{max}(x) = -\beta U\phi'_\mathrm{max}(x),
\end{align*}
so that $U\phi_\mathrm{max}$ is an eigenfunction of $H_{\beta,\tilde\Gamma}$ corresponding to the eigenvalue $\lambda^\delta_1(4\beta^{-1},\tilde\Gamma)$, and repeating the argument from the proof of Corollary~\ref{cor:1even}, we conclude that it is nothing but the ground state of the operator.
\end{proof}

The uniqueness of the maximizer for the comparison operator $H^\delta_{\alpha,\Gamma}$ used in the proof is missing in \cite{EL18} but it can be easily established:
\begin{proposition} \label{prop:uniqdelta}
$\inf\sigma(H^\delta_{\alpha,\Gamma})$ is maximized by $\inf\sigma(H^\delta_{\alpha,\tilde\Gamma})$ uniquely.
\end{proposition}
\begin{proof}
The result follows by Birman-Schwinger principle as in \cite{EL18} where stars of finite edges length were considered and the present $\Gamma$ appears only as a limit; we focus on the differences using the notation of \cite[Example~10.2.1]{EK15}. Since $H^\delta_{\alpha,\Gamma}$ is unitarily equivalent to $\alpha^2 H^\delta_{-1,\Gamma}$ we can without loss of generality consider the last named operator and its Birman-Schwinger counterpart, an integral operator on $\HH=\bigoplus_{j=1}^N L^2(\R_+)$ acting  as
$$
(Q^\kappa\psi)_i(s) = \sum_{j=1}^N \frac{1}{2\pi} \int_0^\infty K_0\big(\kappa d_{ij}(s,s')\big)\, \psi_j(s')\,\D s',
$$
where $d_{ij}(s,s') = \sqrt{s^2+s'^{2} - 2ss'\cos\phi_{ij}}$ with $\phi_{ij}=\theta_i - \theta_j= \sum_{k=ji}^{i-1} \omega_k$ and $K_0$ is the Macdonald function; the largest eigenvalue of $Q^\kappa$ is simple and the value of $\kappa$ for which $Q^\kappa\psi_\mathrm{g}=\psi_\mathrm{g}$ determines the ground state energy of $H^\delta_{\alpha,\Gamma}$ as $-\kappa^2$. As in \cite[Prop.~3.1]{EL18}, the operator $Q^\kappa$ is self-adjoint and positive for each $\kappa>0$, however, it is no longer compact for semiinfinite edges; mimicking the reasoning of \cite[Sec.~5]{EI01} one can check that it has a continuous spectrum in $[0,\frac12\kappa]$ and a nonempty discrete one in $(\frac12\kappa,\infty)$ with the eigenvalues monotonously decreasing towards zero as $\kappa\to\infty$.

Putting $\HH_\mathrm{sym}= \{\psi\oplus \cdots\oplus\psi:\, \psi\in L^2(\R_+)\}\subset\HH$ one finds easily that $Q^\kappa$ is reduced by this subspace and the eigenfunction $\psi_\mathrm{g}$ belongs to $\HH_\mathrm{sym}$, as does the eigenfunction $\tilde\psi$ of $\tilde{Q}^\kappa$. As in \cite{EL18} one checks that $K_0\big(\kappa d_{ij}(s,s')$ is strictly decreasing and strictly convex as a function of the angle $\phi_{ij}$ so using Jensen's inequality we get
\begin{align*}
\sup\sigma(Q^\kappa) & = \frac{1}{2\pi} \int_0^\infty \int_0^\infty \sum_{j=1}^N \sum_{k=1}^N K_0\big(\kappa d_{jk}(s,s')\big)\,\psi_\mathrm{g}(s)\psi_\mathrm{g}(s')\,\D s\D s' \\
& \ge \frac{1}{2\pi} \int_0^\infty \int_0^\infty \sum_{j=1}^N \sum_{k=1}^N K_0\big(\kappa d_{jk}(s,s')\big)\,\tilde\psi(s)\tilde\psi(s')\,\D s\D s' \\
& =\frac{1}{2\pi} \sum_{j=1}^N \int_0^\infty \int_0^\infty \Big(\sum_{k=1}^N K_0\big(\kappa d_{jk}(s,s')\big)\Big)\,\tilde\psi(s)\tilde\psi(s')\,\D s\D s' \\
& > \frac{N}{2\pi} \sum_{j=1}^N \int_0^\infty \int_0^\infty \sum_{k=1}^N K_0\Big(\kappa \Big(\sum_{k=1}^N \frac{d_{jk}(s,s')}{N}\Big)\Big)\,\tilde\psi(s)\tilde\psi(s')\,\D s\D s' \\
& = \frac{1}{2\pi} \sum_{j=1}^N \int_0^\infty \int_0^\infty \Big( \sum_{k=1}^N K_0\big(\kappa \tilde{d}_{jk}(s,s')\big)\Big)\,\tilde\psi(s)\tilde\psi(s')\,\D s\D s' \\
& = \frac{1}{2\pi} \int_0^\infty \int_0^\infty \Big(\sum_{j=1}^N \sum_{k=1}^N K_0\big(\kappa \tilde{d}_{jk}(s,s')\big)\Big)\,\tilde\psi(s)\tilde\psi(s')\,\D s\D s' = \sup\sigma(\tilde{Q}^\kappa)
\end{align*}
unless $\tilde\Gamma$ is congruent with $\Gamma$ which is what we have set out to prove.
\end{proof}

\section{Birman-Schwinger approach to $\delta'$ stars}
\label{s:primestarBS}
\setcounter{equation}{0}

As in Sec.~\ref{s:1Dloop}, there is another approach which is not restricted to even values of $N$; without loss of generality, of course, we may restrict our attention to $N\ge 3$. Recall that the expression for the resolvent of singular Schr\"odinger operators with $\delta'$ interaction supported by curves, and more generally hypersurfaces of codimension one, together with the corresponding Birman-Schwinger principle, was derived in \cite{BLL13}; here we will write it in the form suitable for our purposes.

Consider first a wedge-shape sector $\Omega$ of an opening angle $\omega$ and the outer normal $\nu$. For a given $\kappa>0$ the standard Neumann-to-Dirichlet map of the operator $-\Delta+\kappa^2$ is $\LL_\Omega(\kappa):\,L^2(\partial\Omega)\to L^2(\partial\Omega)$ acting as $\LL_\Omega(\kappa)\psi=u|_{\partial\Omega}$, where $\partial_n u=\psi$; it is bounded, self-adjoint, and it satisfies $\langle\Lambda_\Omega\psi,\psi\rangle=\|\nabla u\|^2_{\Omega}+\kappa^2\|u\|^2_\Omega\ge0$ and $(\LL_\Omega\psi,\psi)=\|\nabla u\|^2_{\Omega}+\kappa^2\|u\|^2_\Omega\ge0$. In the same spirit, we introduce the \emph{Neumann-to-jump operator} on $L^2(\Gamma)=\bigoplus_j L^2(\R_+)$ by $u=\NN_\kappa\phi$ for given $\phi=(\phi_1,\dots,\phi_N)\in L^2(\Gamma)$, where $u$ is the unique solution to the equation
\begin{equation} \label{sectneumann}
(-\Delta+\kappa^2)u=0 \;\;\; \text{ for }\; u\in H^2(\R^2\setminus\Gamma),\quad
  \partial_{\nu_j}^+u=\partial_{\nu_j}^-u-=\phi_j\ \text{ on }\Gamma_j.
\end{equation}
In view of the anticlockwise orientation of the normal field \eqref{normjump}, we have a one-sided sign difference with respect to the outer normals $\nu_j$, so that \eqref{sectneumann} is nothing but the Neumann problem on each sector. The operator
\begin{equation} \label{starBSop}
Q_\kappa:\,L^2(\Gamma)\to L^2(\Gamma), \quad Q_\kappa = -\NN_\kappa,
\end{equation}
is bounded and positive, satisfying $(Q_\kappa\phi,\phi) =\|\nabla\NN_\kappa\phi\|^2_{L^2(\R^2\setminus\Gamma)}+\kappa^2\|\NN_\kappa\phi\|^2$. In a similar way one can introduce $Q_{\sqrt{-z}}$ for any $z\in\C\setminus\R_+$.
\begin{proposition} \label{prop:BSprimestar}
Given $\beta>0$ and a complex $z\not\in \sigma(H_{\beta,\Gamma}) \cap \C\setminus[0,\infty)$, one has
\begin{enumerate}
\item the Krein-type resolvent formula,
$$
(H_{\beta,\Gamma}-z)^{-1}
=R_0(z)+\beta\,\dD_z\,Q_{\sqrt{-z}}(\beta-Q_{\sqrt{-z}})^{-1}\dD_{\bar z}^{*}, \quad
\dD_{\bar z}^{*}h=\partial_\nu R_0(z)h\big|_\Gamma,
$$
where $R_0(z)+(H^{\mathrm{N}}_{\Gamma}-z)^{-1}$ with $H^{\mathrm{N}}_{\Gamma}$ being the Neumann Laplacian in $L^2(\R^2\setminus\Gamma)$,
\item $z\in\sigma_\mathrm{disc}(H_{\beta,\Gamma})$ holds if and only if $\beta\in\sigma_p(Q_{\sqrt{-z}})$ and the respective multiplicities coincide,
$\dim\ker(H_{\beta,\Gamma}-z)=\dim\ker(Q_{\sqrt{-z}}-\beta)$; the corresponding eigenfunctions are related by $\phi\mapsto f=\NN_z\phi$ and $f\mapsto\phi=\partial_\nu f$.
\end{enumerate}
\end{proposition}
\begin{proof}
The argument is essentially the same as in \cite[Thm.~3.11]{BLL13}; the only difference is in the way Krein's formula is written; in the indicated paper it is in terms of the operator $(\beta Q_{\sqrt{-z}})^{-1}$ denoted there $\widehat{M}(\lambda)$.
\end{proof}

Similarly to $Q_\kappa$, the operator $T_\kappa=-Q_\kappa^{-1}$ has a clear meaning, $T_\kappa g=\partial_\nu\dD_\kappa g$ and $\dD_\kappa g$ is the solution with prescribed jump $g$ and continuous
normal derivative. Of the two Krein's formula versions, the one in Proposition~\ref{prop:BSprimestar} suits better our purposes; operator $Q_\kappa$ will play the role analogous to that of the matrix $M_Y$ in the proof of Theorem~\ref{thm:1lambdaN}.

Before proceeding further, recall the scaling properties mentioned in Proposition~\ref{prop:delta'basic} which translate to those of the operator $Q_\kappa$. Explicitly, $Q_\kappa=\kappa^{-1}\mathcal{S}_\kappa Q_1 \mathcal{S}_\kappa^*$ holds for the unitary transformation $(\mathcal{S}_\kappa\psi)(r)=\kappa^{1/2}\psi(\kappa r)$, and consequently, $-\kappa^2\in\sigma_\mathrm{disc}(H_{\beta,\Gamma})$ if and only if $\beta\kappa\in\sigma_\mathrm{disc}(Q_1)$, in other words, the negative spectrum of all the operators $H_{\beta,\Gamma}$ can be expressed in terms of the single operator $Q_1$,
\begin{equation} \label{through_Q1}
  \sigma_\mathrm{disc}(H_{\beta,\Gamma})
  =\Big\{-\frac{\mu^2}{\beta^2}:\ \mu\in\sigma_\mathrm{disc}(Q_1),\ \mu>2\Big\}
\end{equation}
with the multiplicities taken into account; note that $\sigma_\mathrm{ess}(Q_1)=[0,2]$.

To proceed, we need to find the explicit form of the BS operator \eqref{starBSop}. Let us start from a single sector $\Omega_\omega=\{0<\theta<\omega\}$ and look for solution to \eqref{sectneumann} with the outward Neumann data which we now call $\psi_0=-\tfrac1r\partial_\theta u|_{\theta=0}$ and $\psi_\omega=\tfrac1r\partial_\theta u|_{\theta=\omega}$. In polar coordinates, the operator can be written as $r^{-2}(\RR_\kappa -\partial^2_\theta)$ with $\RR_\kappa= \kappa^2 r^2-(r\partial_r)^2$. The Ansatz $u(r,\theta)=R(r)\Theta(\theta)$ leads to $\RR_\kappa R=\lambda R$ and $-\Theta''=-\lambda\Theta$. In contrast to the situation when the region in question has the circular symmetry and $-\partial^2_\theta$ has a discrete spectrum, in case of a sector we obtain for $\lambda=\tau^2\ge 0$ Bessel-type equation $r^2R''+rR'+(\kappa^2r^2-\tau^2)R=0$ with the solution $K_{i\tau}(\kappa r)$, the unique one vanishing as $r\to\infty$; now the `spectral' variable is associated with the radial direction and the angular problem is solved by hyperbolic functions. The functions $K_{i\tau}(\kappa\,\cdot)$ are generalized eigenfunctions of the operator $\RR_\kappa$ which is self-adjoint in $L^2(\R_+,r^{-1}\D r)$. This suggests the Kontorovich--Lebedev transformation \cite[Sec.~12.1]{BE53} mapping the last name space to $L^2(\R_+)$ in the following way,
\begin{equation} \label{KL}
  F(\tau):=\int_0^\infty f(r)K_{i\tau}(\kappa r)\,\frac{\D r}{r},
  \quad
  f(r)=\frac2{\pi^2}\int_0^\infty\tau\sinh(\pi\tau)K_{i\tau}(\kappa r)F(\tau)\,\D \tau,
\end{equation}
with the `reduced' radial part $f(r)=r\psi(r)$ of the function as a natural tool to address our spectral problem; recall that the transformation satisfies a Parseval-type relation, $\int_0^\infty |f(r)|^2\,\frac{\D r}{r} = \frac{2}{\pi^2} \int_p^\infty \tau\sinh(\pi\tau) |F(\tau)|^2\,\D\tau$.

With \eqref{KL} in mind, we look for solution to \eqref{sectneumann} with the given data in the form
$$
u(r,\theta)=\frac2{\pi^2}\int_0^\infty K_{i\tau}(\kappa r)
\Big[A(\tau)\cosh\tau\theta+B(\tau)\cosh\tau(\omega-\theta)\Big]\D\tau .
$$
For any fixed $\tau$ the integrand as a function of $(r,\theta)$ solves the equation in $\Omega_\omega$ and it decays at infinity because $K_{i\tau}(\kappa r)\simeq\sqrt{\pi/(2\kappa r)}\,\ee^{-\kappa r}$ there. The choice of the hyperbolic functions is crucial here because it allows us to choose the coefficients independently at the two parts of the sector boundary from
$$
\partial_\theta u|_{\theta=\omega}
=\frac2{\pi^2}\int_0^\infty K_{i\tau}(\kappa r)\,\tau\sinh(\omega\tau)A(\tau)\,\D\tau,
\;\;
-\partial_\theta u|_{\theta=0}
=\frac2{\pi^2}\int_0^\infty K_{i\tau}(\kappa r)\,\tau\sinh(\omega\tau)B(\tau)\,\D\tau.
$$
By assumption, these values must equal $r\psi_\omega(r)$ and $r\psi_0(r)$, respectively; expressing the latter by means of \eqref{KL}, we conclude that $A(\tau), B(\tau)$ must coincide with $\hat\psi_\omega(\tau)$ and $\hat\psi_0(\tau)$, respectively, up to the multiplicative factor $\frac{\sinh\pi\tau}{\sinh\omega\tau}$, in other words, that
\begin{equation} \label{sector_soln}
u(r,\theta)=\frac2{\pi^2}\int_0^\infty K_{i\tau}(\kappa r)\,
\frac{\sinh\pi\tau}{\sinh\omega\tau}
\Big[\hat\psi_\omega(\tau)\cosh\tau\theta
+\hat\psi_0(\tau)\cosh\tau(\omega-\theta)\Big]d\tau
\end{equation}
is the unique Neumann solution in the class of $H^1$ functions. To get the Neumann-to-Dirichlet map for the sector, we express $u(r,0)$ and $u(r,\omega)$ from \eqref{sector_soln} and insert the function $\hat\psi_0$ and $\hat\psi_\omega$ given by the first relation of \eqref{KL}. We have to distinguish the cases when the two traces are taken from the same halfline boundary of $\Omega$ and from the different ones. This allows us to write the analogue of $Q_\kappa$ for the sector problem as a $2\times 2$ matrix operator, $Q^\Omega_\kappa= \begin{pmatrix}D_\omega&O_\omega\\ O_\omega&D_\omega\end{pmatrix}$, with the diagonal and off-diagonal entries
\begin{align*}
  (D_\omega\psi)(r)&=\int_0^\infty d_\omega(r,\rho)\psi(\rho)\,\D\rho,
  & d_\omega(r,\rho)&=\frac2{\pi^2}\int_0^\infty\sinh(\pi\tau)\coth(\omega\tau)\,
     K_{i\tau}(\kappa r)K_{i\tau}(\kappa\rho)\,\D\tau,\\
  (O_\omega\psi)(r)&=\int_0^\infty o_\omega(r,\rho)\psi(\rho)\,\D\rho,
  & o_\omega(r,\rho)&=\frac2{\pi^2}\int_0^\infty\frac{\sinh\pi\tau}{\sinh\omega\tau}\,
     K_{i\tau}(\kappa r)K_{i\tau}(\kappa\rho)\,\D\tau.
\end{align*}
The kernels are symmetric with a different behavior on the diagonal: $o_\omega$ is smooth, while $d_\omega$ has the $-\frac1\pi\log|r-\rho|$ singularity. This follows the asymptotic behavior of the MacDonald function, $K_{i\tau}(x)\simeq\sqrt{\frac{2\pi}{\tau}}\;\ee^{-\pi\tau/2}\,\sin\!\Big(\tau\log\frac{2\tau}{x}-\tau+\frac\pi4\Big)$ as $\tau\to\infty$. In the kernel of $d_\omega$ the factor $\sinh(\pi\tau)\coth(\omega\tau)\simeq\frac12e^{\pi\tau}$ compensates $e^{-\pi\tau}$ and he integrand behaves like $\frac1{2\tau}\cos(\tau\log\frac\rho r)$ giving rise to the singularity. On the other hand, in $o_\omega$ kernel one has $\frac{\sinh\pi\tau}{\sinh\omega\tau}e^{-\pi\tau}\simeq e^{-\omega\tau}$ and the integral converges absolutely and uniformly together with all derivatives.

What is important for us that for particular values of $\omega$ the kernel expressions allow for a simplification based on a combination of Lebedev's formula \cite[6.794.1]{GR07},
$$
  \int_0^\infty\cosh(\alpha\tau)K_{i\tau}(a)K_{i\tau}(b)\,d\tau
  =\frac\pi2K_0\!\Big(\sqrt{a^2+b^2+2ab\cos\alpha}\Big)
$$
with the identities
$$
\sinh(\pi\tau)\coth\!\Big(\frac{\pi\tau}{m}\Big)=\sum_{k=0}^{m-1}\cosh\!\Big((m-2k)\frac{\pi\tau}m\Big),
\quad
\frac{\sinh\pi\tau}{\sinh(\pi\tau/m)}=\sum_{k=0}^{m-1}\cosh\!\Big((m-1-2k)\frac{\pi\tau}m\Big),
$$
(both follow by summing the geometric series $\sum_{k} \ee^{(m-1-2k)x}$, cf. \cite[1.341-1.342]{GR07} with $x\mapsto ix$) which for $\omega=\pi/m$ yield
\begin{align*}
  d_{\pi/m}(r,\rho) & =\frac1\pi\sum_{k=0}^{m-1}
    K_0\!\Big(\kappa\sqrt{r^2+\rho^2-2r\rho\cos\tfrac{2k\pi}{m}}\Big), \nonumber \\
  o_{\pi/m}(r,\rho) & =\frac1\pi\sum_{k=0}^{m-1}
    K_0\!\Big(\kappa\sqrt{r^2+\rho^2-2r\rho\cos\tfrac{(2k+1)\pi}{m}}\Big). \nonumber
\end{align*}

After this long deviation, let us return to our proper problem. As $\R^2\setminus\Gamma$ is the union of such sectors, we can assemble the sought Neumann-to-jump operator $Q_\kappa$ from the elements constructed above. In the sector $\Omega_j$ the outward Neumann data are $\psi_0=-\phi_j$ on $\Gamma_j$ and $\psi_{\omega_j}=\phi_{j+1}$ on $\Gamma_{j+1}$. Computing the jumps $\jump{u}{j}=u_{\Omega_j}|_{\Gamma_j}-u_{\Omega_{j-1}}|_{\Gamma_j}$, we arrive at the cyclic tridiagonal operator matrix on $L^2(\Gamma)=\bigoplus_{j=1}^N L^2(\R_+)$ with the entries
$$ 
  (Q_\kappa)_{jj}=D_{\omega_{j-1}}+D_{\omega_j},\quad
  (Q_\kappa)_{j,j+1}=(Q_\kappa)_{j+1,j}=-\,O_{\omega_j},\quad
  (Q_\kappa)_{jk}=0\ \text{otherwise},
$$ 
and the indices understood modulo $N$, or explicitly
$$
  Q_\kappa=\begin{pmatrix}
    D_{\omega_N}+D_{\omega_1} & -O_{\omega_1} & & & -O_{\omega_N}\\
    -O_{\omega_1} & D_{\omega_1}+D_{\omega_2} & -O_{\omega_2} & & \\
    & \ddots & \ddots & \ddots & \\
    & & -O_{\omega_{N-2}} & D_{\omega_{N-2}}+D_{\omega_{N-1}} & -O_{\omega_{N-1}}\\
    -O_{\omega_N} & & & -O_{\omega_{N-1}} & D_{\omega_{N-1}}+D_{\omega_N}
  \end{pmatrix}.
$$

In order to find spectral properties of $Q_\kappa$, it is useful to diagonalize the operator. To this aim we note that the transformation \eqref{KL} on $\Omega_\omega$ is bijective but it contains factors which make it asymmetric. We thus introduce the operator\footnote{It is common to use capital letters for operators; using small boldface is a compromise between this habit and an abuse of notation.} $\mathbf{r}$ on $L^2(\R_+)$ of multiplication by $r$, and the operator
$$
(U_\kappa\psi)(\tau)=\frac{\sqrt{2\tau\sinh\pi\tau}}\pi \int_0^\infty\psi(r)K_{i\tau}(\kappa r)\,\frac{\D r}{\sqrt r},
$$
which is unitary as a map $L^2(\R_+,\D r)\to L^2(\R_+,\D\tau)$ by the Parseval-type relation mentioned above. This yields the symmetrized of Kontorovich-Lebedev duality,
\begin{align*}
\hat\psi(\tau) &=\int_0^\infty\psi K_{i\tau}(\kappa r)\,\D r
=\frac{\pi}{\sqrt{2\tau\sinh\pi\tau}}\,(U_\kappa\mathbf{r}^{1/2}\psi)(\tau), \\
(U_\kappa^*\hat\psi)(r) &=\frac1{\sqrt r}\int_0^\infty\frac{\sqrt{2\tau\sinh\pi\tau}}\pi
K_{i\tau}(\kappa r)\hat\psi(\tau)\,\D\tau.
\end{align*}
This allows us to rewrite the action of $D_\omega$ as
$$
(D_\omega\psi)(r)=\frac2{\pi^2}\int_0^\infty\sinh(\pi\tau)\coth(\omega\tau)
K_{i\tau}(\kappa r)\hat\psi(\tau)\,\D\tau =r^{1/2}\big(U_\kappa^*h\big)(r)
$$
with
$$
h=\frac{\pi}{\sqrt{2\tau\sinh\pi\tau}}\cdot\frac2{\pi^2}\sinh(\pi\tau)\coth(\omega\tau)
\cdot\frac{\pi}{\sqrt{2\tau\sinh\pi\tau}}\;U_\kappa\mathbf{r}^{1/2}\psi\,;
$$
the three numerical factors multiply to $\tau^{-1}\coth(\omega\tau)/\tau$ so that
$$
D_\omega=\mathbf{r}^{1/2}U_\kappa^*\,\frac{\coth\omega\tau}{\tau}\,U_\kappa\mathbf{r}^{1/2},\quad
O_\omega=\mathbf{r}^{1/2}U_\kappa^*\,\frac{\csch\omega\tau}{\tau}\,U_\kappa\mathbf{r}^{1/2}.
$$
Returning to the operator $Q_\kappa$ on $L^2(\Gamma)$ we have expressed it as
\begin{equation}\label{Qtransf}
  Q_\kappa=\mathbf{r}^{1/2}\,\bbU_\kappa^{*}\,\frac{M(\tau)}{\tau}\,\bbU_\kappa\,\mathbf{r}^{1/2}
\end{equation}
where $\bbU_\kappa=U_\kappa\otimes I_{\C^N}$ and where $M(\tau)$ is the real symmetric $N\times N$ matrix,
$$
  M(\tau)_{jj}=\coth(\omega_{j-1}\tau)+\coth(\omega_j\tau),\;
  M(\tau)_{j,j+1}=-\csch(\omega_{j}\tau),\;
  M(\tau)_{j,j-1}=-\csch(\omega_{j-1}\tau).
$$
Note that it is $\kappa$-independent; the spectral information enters only through the basis $K_{i\tau}(\kappa r)$ contained in the operator $U_\kappa$. It is not a full diagonalization but it is sufficient for our purpose; now we are in position to state the result.
\begin{theorem}\label{thm:main}
For every $N\ge3$ and $\beta>0$ the lowest eigenvalue $\lambda_1(H_{\beta,\Gamma_\omega})$ exists and is strictly maximized by the equiangular star $\tilde\Gamma$ with $\omega_j=\frac{2\pi}{N},\, j=1,\dots,N$, uniquely up to rotations. The maximal value is $-\frac{\mu_N^2}{\beta^2}$, where
$$
  \mu_N=\sup\sigma\Big(\mathbf{r}^{1/2}U_1^{*}\,\frac{m_N^*(\tau)}{\tau}\,U_1\mathbf{r}^{1/2}\Big)>2
  \quad \text{and} \quad
  m_N^*(\tau)=2\,\frac{\cosh\frac{2\pi\tau}{N}+c_N}{\sinh\frac{2\pi\tau}{N}}.
$$
with $c_N=1$ for even $N$ and $c_N=\cos\frac\pi N$ for odd $N$.
\end{theorem}
\begin{proof}
According to \eqref{through_Q1} it is sufficient to consider $\kappa=1$, by \eqref{Qtransf} we have
$$
\langle\phi,Q_1^{(\omega)}\phi\rangle
=\int_0^\infty\big\langle \hat\Phi(\tau), M_\omega(\tau)\hat\Phi(\tau)\big\rangle\,\frac{\D\tau}{\tau}
\quad \text{with}\;\;\; \hat\Phi=\bbU_1\,\mathbf{r}^{1/2}\phi
$$
for any $\phi\in L^2(\Gamma)$; we add to the operator symbol the label indicating the configuration $\omega=\{\omega_1,\dots,\omega_N\}$. The matrix $M(\tau)$ is the sum of $N\times N$ matrices, each of which contributes through a nonzero $2\times2$ block $\begin{pmatrix}\coth\omega_j\tau&-\csch\omega_j\tau \\ -\csch\omega_j\tau &\coth\omega_j\tau\end{pmatrix}$ acting on the pair $(\hat\Phi_j,\hat\Phi_{j+1})$. They have eigenvalues $(1,\pm 1)$ with the eigenvalues $\tanh\frac{\omega_j\tau}{2}$ and $\coth\frac{\omega_j\tau}{2}$, respectively, which yields
\begin{equation}\label{Msector}
  \big\langle \hat\Phi, M_\omega(\tau)\hat\Phi \big\rangle
  =\sum_{j=1}^N\Big[\tfrac12\coth\tfrac{\omega_j\tau}{2}\,\big|\hat\Phi_j-\hat\Phi_{j+1}\big|^2
  +\tfrac12\tanh\tfrac{\omega_j\tau}{2}\,\big|\hat\Phi_j+\hat\Phi_{j+1}\big|^2\Big]\,;
\end{equation}
note the similarity of this relation with \eqref{qformM_Y}.

Our aim is to compare this form with the one of the symmetric star, $\tilde\omega_j=2\pi/N$. In this case the matrix $M(\tau)$ is circulant, and by the discrete Fourier transform in $j$ one can diagonalize it: we get
$Q_1\simeq\bigoplus_{k=0}^{N-1}Q_1^{(k)}$ on $L^2(\R_+)$, where
$$
  Q^{(k)}_1=\mathbf{r}^{1/2}U_1^{*}\,\frac{m_k(\tau)}{\tau}\,U_1\mathbf{r}^{1/2}
  \quad \text{with}\;\;\;
  m_k(\tau):=2\coth\frac{2\pi\tau}N-2\cos\frac{2\pi k}N\,\csch\frac{2\pi\tau}N.
$$
By $k^*=\lfloor N/2\rfloor$ we denote the index closest to $N/2$ and we put $m_N^*:=m_{k^*}$, that is
$$
  m_N^*(\tau)=2\,\frac{\cosh\frac{2\pi\tau}{N}+c_N}{\sinh\frac{2\pi\tau}{N}},
  \quad
  c_N=-\cos\frac{2\pi k^*}{N}=
  \begin{cases}1,&N\text{ even},\\[1mm]\cos\frac{\pi}{N},&N\text{ odd}.\end{cases}
$$
Since $m_k\le m_{k^*}$ holds for all $k$, we have $Q^{(k)}\le Q^{(k^*)}$ and therefore
$$
  \mu_{\max}(\omega_{\sym})=\sup\sigma\big(Q^{(k^*)}\big)
  =\sup_{\psi\neq0}\frac{1}{\|\psi\|^2}\int_0^\infty m_N^*(\tau)\,|\hat\psi(\tau)|^2\,\frac{\D\tau}{\tau},
  \quad \hat\psi=U_1\mathbf{r}^{1/2}\psi\,;
$$
as expected, the top value is attained twice, for $k^*=\frac{N\pm1}{2}$, if $N$ is odd.

Returning to a general angular configuration $\omega$, we have to choose now trial functions to be inserted into \eqref{Msector}. We take an arbitrary $\psi\in L^2(\R_+)$ and put
$$
  \phi_j=N^{-1/2}\,e^{2\pi i jk^*/N}\,\psi,\quad j=1,\dots,N,
$$
so that $\|\phi\|=\|\psi\|$ and we have
$$
  \big|\hat\Phi_j-\hat\Phi_{j+1}\big|^2=\frac4N\sin^2\!\Big(\frac{\pi k^*}{N}\Big)|\hat\psi|^2,
  \quad
  \big|\hat\Phi_j+\hat\Phi_{j+1}\big|^2=\frac4N\cos^2\!\Big(\frac{\pi k^*}{N}\Big)|\hat\psi|^2.
$$
What is important, with this choice the above values are independent of $j$ and we get
\begin{equation}\label{Mtrial}
  \langle \phi,Q_1^{(\omega)}\phi\rangle
  =\frac2N\int_0^\infty\Big[\sum_{j=1}^N F_N(\omega_j\tau)\Big]|\hat\psi(\tau)|^2\,\frac{\D\tau}{\tau},
  \quad
  F_N(x):=\gamma_N\coth\frac x2+(1-\gamma_N)\tanh\frac x2,
\end{equation}
where
$$
  \gamma_N=\sin^2\frac{\pi k^*}{N}=
  \begin{cases}1,&N\text{ even},\\[1mm]\cos^2\frac{\pi}{2N},&N\text{ odd}.\end{cases}
$$
It is easy to see that $F_N$ is strictly convex on $(0,\infty)$ for every $N\ge2$; for $N$ even it is obvious, for $N$ it is enough to note up to the factor $\frac12$ it coincides with the function $h$ appearing in the concluding part of the proof of Theorem~\ref{thm:1lambdaN}. Since $\sum_j\omega_j=2\pi$, by Jensen's inequality, we then get
$$
  \frac1N\sum_{j=1}^N F_N(\omega_j\tau)\ \ge\ F_N\Big(\frac{2\pi\tau}{N}\Big)
$$
for every $\tau>0$, with equality if and only if all $\omega_j$'s are equal. Inserting this into \eqref{Mtrial} and noting that $2F_N\!\big(\tfrac{2\pi\tau}{N}\big)=m_N^*(\tau)$, we obtain the inequality
$$
  \langle \phi,Q_1^{(\omega)}\phi\rangle
  \ \ge\ \int_0^\infty m_N^*(\tau)|\hat\psi(\tau)|^2\,\frac{d\tau}{\tau}
  =\langle \psi, Q^{(k^*)}_{\sym}\psi\rangle
$$
valid for every $\psi\in L^2(\R_+)$; taking the supremum over all of them, we arrive at
\begin{equation} \label{finalbound}
  \mu_{\max}(\omega)\ \ge\ \sup_{\psi\ne0}\frac{\langle \phi Q_1^{(\omega)}\phi\rangle}{\|\phi\|^2}
  \ \ge\ \mu_{\max}(\tilde\omega)\,;
\end{equation}
this implies the bound we are looking for, $\lambda_1(H_{\beta,\Gamma_\omega})\le\lambda_1(H_{\beta,\Gamma_{\tilde\omega}})$.

It remain to check the existence and uniqueness. The existence requires $m_N^*(\tau)>m_2^*(\tau):=2\coth\frac{\pi\tau}{2}$, where the right-hand side corresponds to the straight line (and the essential spectrum threshold of the original operator). With $y=\pi\tau/N$, this is equivalent to
$$
 \sinh\frac{(N-4)y}{2}+c_N\sinh\frac{Ny}{2}>0,
$$
which is obvious for $N\ge4$; for $N=3$ the left-hand side equals $\tfrac12\sinh\frac y2+2\sinh^3\frac y2>0$. Since the ground state exists, the supremum in \eqref{finalbound} is attained at some $\psi_0$, then it follows from the bijectivity of KL transformation that $\hat\psi_0$ is nonzero on a set of positive measure; the uniqueness is then a consequence of the strict convexity of~$F_N$.
\end{proof}

\section{Concluding remarks}
\label{s:cncl}

A natural question concerns generalization of the result to the situation when the $\delta'$ interaction strength is not constant but changes along the star edges, including the case when the $\Gamma$ is a finite and equilateral star; one can conjecture that the ground state would be again maximized by the configuration symmetric with respect to rotations on the angle $2\pi/N$. For a finite star with $\beta$ constant on the edges one can check that this is the case using the argument analogous to the proof of Theorem~\ref{thm:Neven}. On the other hand, the proof on Theorem~\ref{thm:main} has no straightforward modification because the operator lacks then the scaling property which allowed us to factorize the Birman-Schwinger operator into parts containing the spectral and geometric information separately, cf.~\eqref{Qtransf}.

Going beyond the ground state optimization, one can ask about properties of eigenfunctions of $H_{\beta,\Gamma}$. For instance, while they are generally discontinuous at the edges of $\Gamma$, it makes still sense to speak about their nodal lines, and one may inquire what is the minimum value of $N$ for which the operator has an eigenfunction for which a nodal line is closed. A numerical result \cite{EN03} concerning $\delta$ interaction stars suggests that such an $N$ is not larger than eight.

\section*{Appendix}
\label{s:app}
\setcounter{equation}{0}

Returning to the topic of Proposition~\ref{prop:1repuls}, let $\cY$ be the family of all configurations $Y$. Given $N\ge 2$ and $\alpha<0$, we denote by $\cY_0$ the subset of configurations for which ground state eigenfunction has a stationary point between any two neighboring points of $Y$; recall that the function is positive and concave outside $Y$, hence it can have at most one such point where its derivative vanishes in each subinterval. The symmetric configuration $\tilde Y$ obviously belongs to $\cY_0$.
\begin{proposition} \label{prop:1local}
For any $Y\in\cY_0$ not congruent with $\tilde Y$ we have $\lambda^\delta_1(\alpha,Y) < \lambda^\delta_1(\alpha,\tilde{Y})$.
\end{proposition}
\begin{proof}
We replace the operator $H^\delta_{\alpha,Y}$ by $H^{\delta,\mathrm{N}}_{\alpha,Y}$ in which Neumann conditions are imposed at the stationary points; it is obvious that $\lambda^\delta_1(\alpha,Y) = \inf\sigma(H^\delta_{\alpha,Y}) = \inf\sigma(H^{\delta,\mathrm{N}}_{\alpha,Y})$. The last named operator is a direct sum, $\bigoplus_{j=1}^N H^{\delta,\mathrm{N}}_{\alpha,\hat{I}_j}$, of Neumann Laplacians on intervals between two stationary points containing the interaction sites $y_j$; we use hat to distinguish them from the intervals $I_j$ appearing in Proposition~\ref{prop:1bound} and consider each of them separately. Let us first check that the lowest eigenvalue of $H^{\delta,\mathrm{N}}_{\alpha,\hat{I}_j}$ is maximal if the $\delta$ interaction is placed at the center of $\hat{I}_j$; in view of the natural scaling we can assume without loss of generality that $|\hat{I}_j|=1$. If the interaction position is $a\in(0,1)$, the spectral condition is easily seen to be $k\big(\tan k(1-a)+\tan ka\big) = \alpha$. In view of the strict convexity of the tangent in $(0,\frac12\pi)$, the sum in the bracket is $\ge 2\tan\frac{k}{2}$, sharply so unless $a=\frac12$, and the monotonicity of $k\tan\frac{k}{2}$ yields the result. Using the obvious notation, we thus conclude that $\lambda^\delta_1(\alpha,\hat{I}_j) \le \lambda^\delta_1(\alpha,\tilde{I}_j)$.

Furthermore, a simple scaling argument shows that the lowest eigenvalue of the symmetrically placed $\delta$ interaction is decreasing with the interval $\tilde{I}_j$, which means that to maximize the ground state eigenvalue of the direct sum we have to choose all the $\tilde{I}_j$'s of the same length. However, this is nothing but the operator $H^{\delta,\mathrm{N}}_{\alpha,\tilde{Y}}$, and by Neumann bracketing, its spectral bottom is not larger than $\lambda^\delta_1(\alpha,\tilde{Y})$.
\end{proof}

\medskip

\subsection*{AI usage discloser} The idea of proof of Theorem~\ref{thm:main} was worked out in dialogue with Anthropic Claude Fable on whether the trick from the proof of Theorem~\ref{thm:1lambdaN} can be `lifted' to star graph setting. Its suggestions were checked by the author, who takes full responsibility for their correctness. The AI was also used to verify some calculations in Section~\ref{s:1Dloop}.



\begin{thebibliography}{10}
\bibitem[AGHH]{AGHH}
S.~Albeverio, F.~Gesztesy, R.~H{\o}egh-Krohn, H.~Holden: \emph{Solvable Models in Quantum Mechanics}, 2nd edition,  AMS Chelsea Publishing, Providence, R.I., 2005.
\bibitem[BEL14]{BEL14}
J.~Behrndt, P.~Exner, V.~Lotoreichik: Schr\"odinger operators with $\delta$- and $\delta'$-interactions on Lipschitz surfaces and chromatic numbers of associated partitions, \emph{Rev. Math. Phys.} \textbf{26} (2014), 1450015
\bibitem[BLL13]{BLL13}
J.~Behrndt, M.~Langer, V.~Lotoreichik: Schr\:odinger operators with $\delta$ and $\delta'$-potentials supported on hypersurfaces, \emph{Ann.~H.~Poincar\'{e}} \textbf{14} (2013), 385--423.
\bibitem[Da79]{Da79}
P.J.~Davis: \emph{Circulant matrices}, Wiley-Interscience, New York-Chichester-Brisbane 1979.
\bibitem[BE53]{BE53}
A.Erd\'elyi, ed.: \emph{Higher Transcendental Functions}, vol. II, Based, in part, on notes left by Harry Bateman, McGraw-Hill, New York 1953.
\bibitem[Ex19]{Ex19}
P.~Exner: An optimization problem for finite point interaction families, \emph{J. Phys. A: Math. Theor.} \textbf{52} (2019), 405302
\bibitem[EI01]{EI01}
P.~Exner, T.~Ichinose: Geometrically induced spectrum in curved leaky wires, \emph{J. Phys. A: Math. Gen.} \textbf{34} (2001), 1439--1450.
\bibitem[EK15]{EK15}
P.~Exner, H.~Kova\v{r}\'{\i}k: \emph{Quantum Waveguides}, Springer, Cham 2015.
\bibitem[EL18]{EL18}
P.~Exner, V.~Lotoreichik: Optimization of the lowest eigenvalue for leaky star graphs, in  in Proceedings of the conference `Mathematical Results in Quantum Physics' (QMath13, Atlanta 2016; F.~Bonetto, D.~Borthwick, E.~Harrell, M.~Loss, eds.), Contemporary Mathematics, AMS, Providence, R.I. 2018; p. 187--196.
\bibitem[EN03]{EN03}
P.~Exner, K.~N\v{e}mcov\'{a}: Leaky quantum graphs: approximations by point inter\-action Hamiltonian, \emph{J. Phys. A: Math Gen.} \textbf{36} (2003), 10173--10193.
\bibitem[GR07]{GR07}
I.S.~Gradshteyn, I.M.~Ryzhik: {\em Tables of Integrals, Series, and Products}, 7th edition (by A.~Jeffrey and D.~Zwillinger), Academic Press, Amsterdam 2007
\bibitem[Kr46]{Kr46}
M.G.~Krein: On resolvents of a symmetric operator with deficiency index $(m,m)$, \emph{Sov. Math. Doklady} \textbf{52} (1946), 657--660.
\bibitem[PV22]{PV22}
K.~Pankrashkin, M.~Vogel: On Schr\"odinger operators with $\delta'$-potentials supported on star graphs, \emph{J. Phys. A: Math. Theor.} \textbf{55} (2022), 295201
\bibitem[RS78]{RS78}
M.~Reed, B.~Simon: \emph{Methods of Modern Mathematical Physics IV.~Analysis of Operators}, Academic Press, New York 1978.
\bibitem[We80]{We80}
J.~Weidmann: \emph{Linear Operators in Hilbert Space}, Springer, Heidelberg 1980.

\end{thebibliography}
\end{document}